\documentclass[12pt]{article}
\usepackage{amsthm}
\usepackage{amsmath}
\usepackage{amsfonts}
\usepackage{graphicx}
\usepackage{latexsym}
\usepackage{amssymb}
\usepackage{appendix}
\usepackage{algorithm2e}

\usepackage{comment}
\usepackage{array,booktabs,longtable}

\usepackage[hidelinks]{hyperref}
\usepackage[nameinlink,capitalise,noabbrev]{cleveref}
\usepackage{xcolor}
\newcommand{\abs}[1]{\left|#1\right|}

\newcommand{\field}[1]{\mathbb{#1}}

\newcommand{\dS}{\field{S}}

\newcommand{\cC}{{\cal C}}
\newcommand{\cD}{{\cal D}}
\newcommand{\cE}{{\cal E}}

\newcommand{\cS}{{\cal S}}
\newcommand{\cP}{{\cal P}}
\newcommand{\cQ}{{\cal Q}}

\newcommand{\bldx}{{\mathbf{x}}}
\newcommand{\bldy}{{\mathbf{y}}}

\newcommand{\nbinw}{\binom{[n]}{w}}

\newcommand{\mmod}{{\mbox{mod}}}
\newcommand{\nminusw}{\bar{\omega}}
\newcommand{\lexpo}{s_\varnothing}

\newtheorem{defn}{Definition}
\newtheorem{theorem}{Theorem}
\newtheorem{lemma}{Lemma}
\newtheorem{corollary}{Corollary}

\newtheorem{proposition}{Proposition}

\numberwithin{theorem}{section}
\numberwithin{lemma}{section}
\numberwithin{corollary}{section}
\numberwithin{proposition}{section}

\begin{document}

\bibliographystyle{plain}

\title{There Are No Nontrivial\\ Perfect Codes in the Johnson Scheme}
\author{
{\sc Tuvi Etzion}\thanks{Technion Israel Institute of Technology, Henry and Marilyn Taub Faculty of Computer Science,
Haifa 3200003, Israel, e-mail:{\tt etzion@cs.technion.ac.il}.} \and
{\sc Xiande Zhang}\thanks{School of Mathematical Sciences
University of Science and Technology of China, Hefei, Anhui, 230026, China,
and Hefei National Laboratory, University of Science and Technology of China, Hefei, 230088, China,
e-mail:{\tt drzhangx@ustc.edu.cn}.} \and
{\sc Wenjie Zhong}\thanks{School of Mathematical Sciences
University of Science and Technology of China, Hefei, Anhui, 230026, China,
e-mail:{\tt zhongwj@mail.ustc.edu.cn}.}}
\maketitle

\begin{abstract}
In his seminal work from 1973, Delsarte conjectured that the Johnson schemes admit no
nontrivial perfect codes. We prove this fundamental conjecture which is central in combinatorial coding theory.
Our arguments combine lower bounds on the design strength of a perfect code with divisibility
conditions for the size of a metric ball. We show that the ball size divides the greatest common
divisor of specified binomial coefficients. An exact formula for this gcd implies squarefreeness and
localizes a unique power of each prime divisor of the ball size to a short interval. Lloyd polynomials,
explicit prime estimates, and bounds for linear forms in logarithms then exclude all radii except for $1$ and $2$.
Additional arithmetic arguments handle these two remaining radii.
\end{abstract}

\section{Introduction}

Codes that attain a sphere packing bound are called perfect.
Such codes have always drawn the attention of coding theorists and mathematicians.
In the Hamming scheme, all the parameters of perfect codes over finite fields are known~\cite[Chapter 6]{McSl77}.
They exist for only a relatively small number of parameters, while
elsewhere it was proved \cite{vLin,McSl77,Tie73,ZiLe73} that they cannot exist.
The nonexistence proof is based on Lloyd's polynomials~\cite{Llo57}. For non-field
alphabets only trivial perfect codes are known and by similar methods it
was proved \cite{Best83} that for most other parameters they do not exist, with the possible exception of radii $1$, $2$, $6$ and $8$.
Hong~\cite{Hon84} subsequently excluded radii $6$ and $8$.
Recently, Bennett~\cite{Ben26} excluded broad families of parameters where perfect codes with radius $2$ cannot exist.
The seminal work of Delsarte~\cite{Del73} connected association schemes with coding theory.
The Hamming and Johnson schemes are the most important schemes in coding theory~\cite{AAK01,Del73},\cite[Chapter 21]{McSl77}.
For most of the schemes of importance in coding theory, the existence
problem for perfect codes is settled~\cite{Etz22}.

The Johnson scheme J$(n,w)$, whose distance-one graph is the Johnson graph, is important in both coding theory
and combinatorics. The elements can be represented by binary words of length $n$ and weight $w$. The distance, $d(\bldx,\bldy)$, between
two such words $\bldx=(x_1,x_2,\ldots,x_n)$ and $\bldy=(y_1,y_2,\ldots,y_n)$ is half of their Hamming distance, i.e.,
$$
d(\bldx,\bldy)= \abs{ \{ i ~:~ x_i=1 ~ \text{and}~ y_i=0, ~ 1 \leq i \leq n \} }= \abs{ \{ i ~:~ x_i=0 ~ \text{and}~ y_i=1, ~ 1 \leq i \leq n \} }.
$$
This framework for the Johnson scheme can be described in terms of sets. The elements of J$(n,w)$
(also, vertices of the graph) are all the $w$-subsets of an $n$-set, $[n] \triangleq \{ 1,2,\ldots,n\}$, and the distance, $d(X,Y)$,
between two such subsets $X$ and $Y$ is
$$
d(X,Y)= \abs{X \setminus Y} = \abs{Y \setminus X}.
$$
We use the notation $[a,b] \triangleq \{ a,a+1,\ldots,b\}$ and $(a,b] \triangleq \{ a+1,\ldots,b\}$, where $a$ and $b$ are integers.

We recall several standard parameters, results, and properties
of the Johnson scheme that are used for dealing with perfect codes.
An $e$-ball (sometimes called an $e$-sphere)
centered at a $w$-subset $X$ is the set of all $w$-subsets in J$(n,w)$
that are at distance at most $e$ from~$X$. The size of such a ball, $\Phi_{e}(n,w)$, is
\begin{equation}
\label{eq:ball_size}
\Phi_{e}(n,w) = \sum_{i=0}^e \binom{w}{i} \binom{n-w}{i} .
\end{equation}
An $e$-perfect code in J$(n,w)$ is a set $\cC$ of $w$-subsets in J$(n,w)$ such that the $e$-balls
around the codewords of $\cC$ form a partition of J$(n,w)$. In other words, for each word $X$ in J$(n,w)$, there exists
exactly one codeword $Z \in \cC$ such that $d(X,Z) \leq e$.
We say that such a code $\cC$ has radius $e$.
Hence, by the sphere-packing bound, the number of codewords of an $e$-perfect code $\cC$ in J$(n,w)$ is
\begin{equation}
\label{eq:code_size}
\abs{\cC} = \frac{\binom{n}{w}}{\Phi_{e}(n,w)}.
\end{equation}
This implies that if there exists an $e$-perfect code in J$(n,w)$, then $\Phi_{e}(n,w)$ divides $\binom{n}{w}$.
This is a very basic property, and we will see that it can be generalized so that $\Phi_{e}(n,w)$
must divide an extensive range of binomial coefficients, where $\binom{n}{w}$ is a special case.
As a consequence, the divisibility properties of $\Phi_{e}(n,w)$
will be a major tool in proving the nonexistence of perfect codes in J$(n,w)$.
These divisibility conditions are associated with the prime factorization of $\Phi_e (n,w)$.

As noted above, the connection between association schemes and coding theory was made in the seminal work of Delsarte~\cite{Del73}.
Except for trivial perfect codes defined for all metrics (the whole space is $0$-perfect and one codeword
is $w$-perfect), he noticed that another family of trivial perfect codes
exists in J$(2w,w)$, where $w=2k+1$, and the radius of the code is $e=k$.
The code consists of only two codewords, a $w$-subset $X$ and its complement $[n] \setminus X$.
He conjectured that these are the only perfect codes in J$(n,w)$.
In his pioneering work from 1973~\cite{Del73}, Delsarte wrote on page 55 (the following lines are the exact quote):

\begin{quote}
``After having recalled there are ``very few'' perfect codes
in the Hamming schemes, one must say that, for $1 < \delta < n$,
there is not a single one known in the Johnson schemes. It is tempting to
risk the conjecture that such codes do not exist. Certain results
contained in the present work could be useful to attack this
problem; especially the generalized Lloyd theorem of sec. 5.2.2
and theorem 4.7 about $t$-designs.''
\end{quote}
Delsarte could not prove his conjecture, and many attempts to solve it in the years that followed
were only partially successful. In his plenary talk at the International Symposium on Information Theory, held in Seattle during July 2006,
Alexander Vardy mentioned the existence problem of nontrivial perfect codes in the Johnson scheme as
one of the major open problems remaining to be solved in coding theory. From this point on,
all perfect codes are assumed to be nontrivial, i.e., $1 \leq e \leq w-1$
and when $n=2w$ we have $w \neq 2e+1$. The current work solves the problem completely.
Hence, the main result of the paper is given in the following theorem.

\begin{theorem}
\label{thm:main_theorem}
There are no nontrivial perfect codes in $\textup{J}(n,w)$.
\end{theorem}

The main technique used in this paper is to consider the $e$-ball of the Johnson scheme J$(n,w)$.
The size of this ball must divide the size of the space $\binom{n}{w}$.
It is proved that, if there exists an $e$-perfect code in J$(n,w)$,
then the size of the $e$-ball, $\Phi_e(n,w)$, is squarefree.
Moreover, if there exists an $e$-perfect code in J$(n,w)$, then
there exists a very short interval such that a unique power of each prime divisor of $\Phi_e(n,w)$ is contained in this interval.
We consider the sum of reciprocals of the exponents of those powers of primes.
The short length of the interval containing those powers of primes strongly
restricts the distribution of these exponents.
Combining some results on distribution of primes and Lloyd's theorem (used to exclude
perfect codes in association schemes) for the Johnson scheme,
we obtain incompatible upper and lower bounds on the sum of reciprocals, and thus exclude the existence
of $e$-perfect codes for $e \geq 3$. Finally, for radius $1$, explicit estimates for the sum of reciprocals are used,
and for radius $2$, the theory of simple continued fractions is applied.

The remainder of the paper and a road-map for the solution are
organized as follows. \cref{sec:survey} presents some preliminary definitions and known results
required for our exposition. It also briefly surveys some of the known techniques associated with this problem.
\cref{sec:strength} is devoted to excluding possible perfect codes based on designs that are either formed by the code
or embedded in the code. In particular, we
examine the strength $\varphi$ of an $e$-perfect code, where $e \geq 2$. Greater strength imposes more divisibility
conditions on the size of an $e$-ball. Previously, the bound $\varphi \geq \lfloor \frac{w}{2} \rfloor$ was known for
$e \geq 2$ when $w$ is sufficiently large, and in this section strength of at least $w - e \sqrt{w}$,
is proved for all $w$ when $e \geq 2$ (\cref{thm:regular for all w}).
As a consequence, we show that if there exists an $e$-perfect code in J$(n,w)$,
then $\Phi_e(n,w)$ must be squarefree when $e \geq 2$ (\cref{thm:sufficient condition}).
Moreover, the strength redundancy $d=w-\varphi$ is small and determines the length of the short interval mentioned before.
An immediate conclusion is that if $p^2$ divides $e+1$ for some prime $p$, then there are no $e$-perfect codes (\cref{cor:square-factor-e-plus-one}).

The approach in \cref{sec:divide_ball,sec:JohnsonAlgebra,sec:complete-exclusion} is an estimate of
the sum of reciprocals of the exponents of primes that divide $\Phi_e(n,w)$ in the short interval.
\cref{sec:divide_ball} deals with the primes that divide the size of the ball $\Phi_{e}(n,w)$.
The section starts with two auxiliary lemmas used
to show that in an interval of length $e/3$ above $e$ every prime must divide the size of the ball $\Phi_{e}(n,w)$
if an $e$-perfect code exists in J$(n,w)$ (\cref{cor:moving-prime-detector}).
After that it presents a theorem
implying a sequence of binomial coefficients that must be divisible by $\Phi_{e}(n,w)$ (\cref{thm:mixed-design-divisibility}).
As a consequence the size of the $e$-ball must also divide the gcd (greatest common divisor) of these binomial coefficients.
Finally, it evaluates the exact primes that participate in this gcd and proves that this gcd is squarefree, implying the previous result
that the size of the ball is also squarefree (\cref{cor:prime factors from chain divisitilities}) and this result also applies to $e=1$.
This result also shows that each prime divisor of $\Phi_e(n,w)$ has a power
which is contained in the interval $(n-w,n-w+d]$, where $d$ is the strength redundancy.
It motivates the idea of considering the sum of reciprocals.
In \cref{sec:JohnsonAlgebra}, the algebra of the Johnson association scheme is presented, and in particular attention to
Lloyd's polynomials, which have a crucial role in proving the nonexistence
of perfect codes. The main result is an adaptation of Lloyd's theorem
for the Johnson scheme (\cref{thm:section5-Lloyd-factorization}).
The estimate of the sum of reciprocals is done in
\cref{sec:complete-exclusion}, which also presents some results concerning the distribution
of primes. It takes the divisibility theorems concerning primes which were derived in previous sections
to cap the localized exponents to deduce upper bounds on the sum of reciprocals,
and combines them with Lloyd's polynomials to give incompatible lower bounds of the sum of reciprocals.
It shows that when $e \geq 3$ there are no $e$-perfect
codes in the Johnson scheme (\cref{thm:section5-unified-tail}).

The rest of the paper is devoted to radii $1$ and $2$.
Section~\ref{sec:no_1_perfect} shows that there are no $1$-perfect codes
in the Johnson scheme. We exclude this radius, long considered the most difficult, using
the previous techniques and further analysis on the structure of $\Phi_1(n,w)$.
\cref{sec:no_2_perfect} uses continued fractions to exclude perfect codes with radius $2$.
This case can also be treated by the earlier methods, but that proof is longer.
As a consequence of the main exclusion theorems
(\cref{thm:section5-unified-tail}, \cref{thm:no-one-perfect-codes-optimized}, and \cref{thm:section5-two})
in all the sections we conclude the proof of \cref{thm:main_theorem}.
All these results and the combination of all the methods mentioned in this summary form the novel contributions of this work.

\section{Survey of Known Techniques and Results}
\label{sec:survey}

A simple observation is that if there exists an $e$-perfect code $\cC$ in J$(n,w)$, then its
complement code $\cC^\text{c} \triangleq \{ [n]\setminus Z ~:~ Z \in \cC \}$,
is an $e$-perfect code in J$(n,n-w)$. This implies that we only have to consider schemes of the form J$(n,w)$, where $n \geq 2w$,
so this will be our basic assumption throughout our presentation. For the remaining part of the paper, we write $n=2w+\delta$
and $\nminusw \triangleq n-w = w + \delta$.

Since Delsarte made his conjecture regarding the Johnson scheme, there have been many attempts to solve his conjecture.
Biggs~\cite{Big73} followed Delsarte and established a criterion for the existence of perfect codes
in distance-transitive graphs (the Johnson graph is distance-transitive).
This criterion implies Lloyd's theorem (presented in \cref{sec:JohnsonAlgebra}) which was used to prove the nonexistence
of perfect codes in the Hamming scheme. His work motivated further research and
the first major attempt was made by Bannai~\cite{Ban77} who tried to use Lloyd's theorem as suggested
by Delsarte. He proved that there are no $e$-perfect codes in J$(2w+1,w)$ for $e \geq 2$.
Hammond~\cite{Ham82} continued with a similar approach and proved that there are no perfect codes in J$(2w+2,w)$
and J$(2w+1,w)$. Ten years after Delsarte's conjecture Roos~\cite{Roo83} proved a significant result.
Using a code-anticode theorem (a~generalization of the sphere-packing bound),
by observing that $e$-balls must be maximum size anticodes with diameter~$2e$, Roos proved the bound now known as the Roos bound: if
there exists an $e$-perfect code in J$(n,w)$, then
\begin{equation}
\label{eq:Roos}
n \leq (w-1) \frac{2e+1}{e}.
\end{equation}

Martin's doctoral thesis~\cite{Mar92} contains a comprehensive study of completely regular subsets of Johnson graphs.
He used this theory to study perfect codes in the Johnson graph. He ruled out many possible existence parameters of $1$-perfect
codes and $2$-perfect codes. This was done for example by examining solutions to some Diophantine equations.
He also found the exact strength of $1$-perfect codes.
A completely different approach was presented in~\cite{Etz96}. A partition of the coordinates of a possible perfect code
into two subsets, was used. On each of these two subsets the weight distribution of the codewords was examined, as well as Steiner
systems embedded in these parts. This made it possible to exclude many schemes (graphs) in which perfect codes cannot exist.
For example, there are no perfect codes in J$(2w+p,w)$, where $p$ is a prime,
no perfect codes in J$(2w+2p,w)$, where $p$ is a prime different from~$3$, or
no perfect codes in J$(2w+3p,w)$, where $p$ is a prime larger than $6$.
The Steiner systems embedded in these $e$-perfect codes imply that $e$, $w$, and $n$ must obey certain modulo conditions.
Further results in this direction were proved in~\cite{Etz01} and the partition of the coordinates was
used for a different proof of \cref{eq:Roos}.
Another step forward using similar methods was taken by Shimabukuro~\cite{Shi05} who proved that there are no perfect codes in
J$(2w+5p,w)$, where $p$ is a prime different from $3$, and in J$(2w+p^2 ,w)$, where $p$ is a prime.

A major advance was made in~\cite{EtSc04}. First, more Steiner systems and divisibility conditions
were proved and as a consequence it was proved that \cref{eq:Roos} cannot be met with equality, i.e.,
\begin{equation*}
n < (w-1) \frac{2e+1}{e},
\end{equation*}
which implies that
\begin{equation}
\label{eq:section5-w-over-z}
\frac{w}{\nminusw} >\frac {e}{e+1}.
\end{equation}
Next, while previous results excluded certain graphs, presented trade-offs
between $n$, $w$, and $e$, or proved the existence of some designs embedded in the code, the main technique used in~\cite{EtSc04} is completely different.
It also uses a partition of the coordinates into two subsets and it also looks for designs of the codewords on these two subsets.
However, it considers only the strength of the design and eventually excludes possible radii in which perfect codes cannot exist.
The method formulates the problem in terms of polynomial equations and applies number-theoretic arguments.
In the process also many graphs were shown to admit no perfect codes.
The proofs used congruences and factorizations of binomial coefficients.
Two old theorems in number theory were used, Kummer's theorem~\cite[p. 245]{GKP94,Gra97} and Lucas' theorem~\cite{Fin47,Luc91}.
As these theorems are going to be used extensively in our exposition, they will be stated here.

\begin{theorem} [Lucas' theorem]
\label{thm:Lucas}
Let $p$ be a prime and let
$$
k = \sum_{j \geq 0} k_j p^j,  ~~~ r = \sum_{j \geq 0} r_j p^j, ~~~ 0 \leq k_j , r_j <p .
$$
Then
$$
\binom{k}{r} \equiv \prod_{j \geq 0} \binom{k_j}{r_j} ~ (\mmod ~ p).
$$
Moreover, $p$ does not divide $\binom{k}{r}$ if and only if $r_j \leq k_j$ for every $j$.
\end{theorem}

\begin{theorem} [Kummer's theorem]
\label{thm:Kummer}
Let $p$ be a prime. The number of times $p$ appears in the factorization of $\binom{k}{r}$ equals the number of carries when adding
$r$ to $k-r$ in base $p$.
\end{theorem}

In~\cite{Etz06} further considerations were done on the
designs embedded in possible perfect codes. It was claimed that the most difficult cases
to exclude possible $e$-perfect codes in J$(n,w)$ are when $e=1$ and when $n=2w$.
Gordon~\cite{Gor06} concentrates on $1$-perfect codes in J$(n,w)$. He also used number theory and proved the following result which will
be essential in Section~\ref{sec:no_1_perfect}.
\begin{proposition}[Theorem 8 in~\cite{Gor06}]
\label{prop:no1large}
For $n < 2^{250}$ there are no $1$-perfect codes in $\textup{J}(n,w)$.
\end{proposition}
The directions in number theory used for $1$-perfect codes by Gordon~\cite{Gor06} can be compared with more extensive work
with number theory used in our paper for all radii. Although direct generalizations were not done, some of his ideas for radius $1$
were used for all radii in the current work. This includes squarefreeness of the $e$-ball when an $e$-perfect code exists, powers of primes
in a short interval, and the sum of reciprocals of the associated exponents.
Silberstein~\cite{Sil07} (see also~\cite{SiEt10}) improved the
bound in \cref{eq:Roos} for $1$-perfect codes and proved that if a $1$-perfect code exists in J$(2w+\delta,w)$, then $\delta < w/11$.
She also eliminated $2$-perfect codes in J$(2w,w)$ for $5 < w \leq 1.97 \cdot 10^{7655}$ using another concept
from number theory, namely Pell's equations~\cite{JaWi09}.
Finally, Bannai and Noda~\cite{BaNo16} used their enumeration work on counting the number of blocks in a design
to obtain a nice improvement to the Roos bound in \cref{eq:Roos} for $e \geq 2$.

Before starting the presentation on the road to solving Delsarte's long-standing conjecture
we want to emphasize that perfect codes in the Johnson scheme should not be confused with perfect constant-weight codes~\cite[Chapter 9]{Etz22}.
These codes are also rare, but exist over some non-binary alphabets~\cite{EtvL01,vLTo99,Sva99}.
In the binary case these two types of codes coincide.

\section{Regularity of Codes and Embedded Designs}
\label{sec:strength}

In this section we introduce the method suggested in~\cite{EtSc04} to eliminate possible
$e$-perfect codes in J$(n,w)$.
The results in~\cite{EtSc04} are derived using the regularity of an $e$-perfect code which is
also associated with the $t$-designs that the code admits. The largest integer $t$ for which the code forms
a $t$-design is its strength. As the strength increases, the size of an $e$-ball in J$(n,w)$ must divide
some binomial coefficients. If it is proved that it cannot divide one of the coefficients, it implies
that there is no $e$-perfect code in J$(n,w)$.

Let $\cC$ be  a code in J$(n,w)$, $A\subseteq [n]$, and $I\subseteq A$. Define
$$
\cC_{A}(i)=\Big| \{Z\in\cC ~:~ |Z \cap A |=i\} \Big|
$$
and
$$
\cC_{A}(I)=\Big| \{Z\in\cC ~:~ Z \cap A=I\} \Big|.
$$

\begin{defn}
\label{def:ed} A code $\cC$ in $\textup{J}(n,w)$ is said to be {\em $t$-regular} if the following two conditions are satisfied.

\noindent {\bf (c.1)} There exist numbers $\alpha (0) ,..., \alpha (t)$ such that if $A \subseteq [n]$, $\abs{A} =t$, then
$\cC_{A} (i) = \alpha (i)$ for all $0 \leq i \leq t$. In other words, the
number $\cC_{A} (i)$ is independent of the set $A$ for every $0 \leq i \leq t$.

\noindent {\bf (c.2)} For any given $t$-subset $A$ of $[n]$, there exist numbers $\beta_{A} (0) ,..., \beta_{A} (t)$ such
that if $I \subseteq A$, then $\cC_{A} (I) = \beta_{A} (\abs{I})$. In other words, the number $\cC_{A} (I)$ depends only on $\abs{I}$.
\end{defn}

Note that if $\cC$ is $t$-regular, then it is also $t'$-regular for all $t' < t$.

\begin{defn}
A \emph{$t$-design} $S_{\lambda} (t,w,n)$ is a collection $\cC$ of
$w$-subsets, called \emph{blocks}, of $[n]$, such that each
$t$-subset of $[n]$ is a subset of exactly $\lambda$ blocks of
$\cC$. When $\lambda =1$ the design is called a \emph{Steiner system} and is denoted by $\textup{S}(t,w,n)$. The largest $t$ of a code
$\cC$ for which the code is a $t$-design is called the \emph{strength} of the code and will be denoted by $\varphi$.
\end{defn}

 The following proposition was proved in~\cite{Etz06}.
\begin{proposition}
A code $\cC$ in \textup{J}$(n,w)$ is $t$-regular if and only if it forms a $t$-design.
\end{proposition}

There are simple well-known necessary divisibility conditions for the existence of a Steiner system S$(t,w,n)$.
\begin{proposition}[Corollary 3.1 on page 49 in \cite{Etz22}]
\label{prop:necessary_Steiner}
If there exists a Steiner system \textup{S}$(t,w,n)$, then for all $0 \leq i \leq t$, the numbers
$$
\binom{n-i}{t-i} \Big{/} \binom{w-i}{t-i}
$$
must be integers.
\end{proposition}

In~\cite{Etz96} it was proved that many Steiner systems are embedded in an $e$-perfect code.
\begin{proposition} [Theorem 3.1 and Corollary 3.2 in~\cite{Etz96}]
\label{prop:embedded_Steiner}
If there exists an $e$-perfect code in \textup{J}$(n,w)$, then there exist Steiner systems \textup{S}$(e+1,2e+1,w)$ and \textup{S}$(e+1,2e+1,\nminusw)$.
\end{proposition}

More systems were found in~\cite{Etz01,EtSc04}. The following four systems were found to be embedded if
an $e$-perfect code exists in J$(n,w)$:

$$
\textup{S}(2,e+2,w+2), ~~~ \textup{S}(2,e+2,\nminusw+2),
$$
$$
\textup{S}(2,e+2,w-e+1), ~~~ \textup{S}(2,e+2,\nminusw-e+1).
$$
This implies the following important divisibility condition summarized in~\cite{EtSc04}.
\begin{proposition}
\label{prop:four_divide}
If there exists an $e$-perfect code in \textup{J}$(n,w)$, then
\begin{align*}
(e+1)(e+2)&\mid (w+1)(w+2),\\
(e+1)(e+2)&\mid (\nminusw+1)(\nminusw+2),\\
(e+1)(e+2)&\mid (w-e)(w-e+1),\\
(e+1)(e+2)&\mid (\nminusw-e)(\nminusw-e+1).
\end{align*}
\end{proposition}

Since
$$
(w+1)(w+2)-(w-e)(w-e+1) =(e+1)\bigl(2(w+2)-(e+2)\bigr),
$$
the first and third divisibility conditions of \cref{prop:four_divide} imply
that $e+2\mid 2(w+2)$. Applying the same argument to $\nminusw$ with the second and fourth
divisibility conditions of \cref{prop:four_divide}, we obtain that $e+2\mid 2(\nminusw+2)$ and hence we have
\begin{equation}
\label{eq-congruence-e+2}
2(w+2)\equiv 2(\nminusw+2)\equiv 0~(\mmod ~e+2).
\end{equation}

Martin~\cite{Mar92} examined the strength of $1$-perfect codes in J$(n,w)$ and proved the following result.
\begin{proposition}
\label{prop:Martin}
If $\cC$ is a $1$-perfect code in \textup{J}$(n,w)$, then there exist positive integers $r$ and $k$, such that $r \geq k+1$, and
the weight $w$, length $n$, and strength~$\varphi$ are
$$
w=kr+1, \qquad n=2kr+r-k+1, \qquad \varphi=k(r-1).
$$
\end{proposition}
The strength of \cref{prop:Martin} was written differently in~\cite{EtSc04} and in a third way in~\cite{Etz06}.
When $e \geq 2$ it was proved in~\cite{EtSc04} that the strength is at least $\lfloor w/2 \rfloor$
and this is only if $w$ is large enough. This result will be considerably improved now. For the following polynomial defined
in~\cite{EtSc04},
\begin{equation}
\label{eq:the_key_polynomial}
\sigma_e(w,\delta,t)=\sum_{i=0}^{e}\sum_{j=0}^{i}
(-1)^j \binom{t}{j}\binom{w-j}{i-j}\binom{w+\delta-t+j}{i},
\end{equation}
the following criterion was proved.

\begin{proposition} [Theorem 18 in~\cite{EtSc04}]
\label{prop:main_t_regular}
Let $\cC$ be an $e$-perfect code in \textup{J}$(2w+\delta,w)$ and let ${1 \leq t \leq w}$.
If $\sigma_e(w,\delta,m) \neq 0$ for all integers $1 \leq m \leq t$ then $\cC$ is $t$-regular.
\end{proposition}

Some consequences of \cref{prop:main_t_regular} given in~\cite{EtSc04} are the following two propositions.
\begin{proposition} [Corollary 11 in~\cite{EtSc04}]
\label{prop:p2_divides_ball}
If $e \geq 2$ and $p$ is a prime, then there is no $\lfloor w/2\rfloor$-regular $e$-perfect code in \textup{J}$(n,w)$ when
$p^2$ divides $\Phi_e(n,w)$.
\end{proposition}

\begin{proposition} [Theorem 27 in~\cite{EtSc04}]
\label{prop:no-perfect-3}
If $p$ is a prime, $e\equiv-1\pmod{p^2}$, and an $e$-perfect code exists in \textup{J}$(n,w)$, then
$p^2$ divides $\Phi_e(n,w)$.
\end{proposition}

\begin{proposition} [Corollary 10 and the proof of Theorem 22 in~\cite{EtSc04}]
\label{prop:lastEtSc}
If there exists an $e$-perfect code in $\textup{J}(n,w)$, then it is $e$-regular and $\sigma_e(w,\delta,m) > 0$ for all $1 \leq m \leq e$.
\end{proposition}

In the next few results we improve on the result in~\cite{EtSc04} regarding the strength of the code,
i.e., $\varphi \geq \lfloor w/2 \rfloor$ for large enough $w$. We first prove a new lower bound
on the strength of the code. This bound is usually better than the lower bound of $\lfloor w/2 \rfloor$
proved in~\cite{EtSc04}. Moreover, our bound holds for every admissible $w$, whereas~\cite{EtSc04} required $w$ to be sufficiently large.
\begin{theorem}
\label{thm:regular for all w}
If $\cC$ is an $e$-perfect code in \textup{J}$(n,w)$, where $e \geq 2$, then $\cC$
is $\bigl(w-\lfloor e\sqrt{w}\rfloor\bigr)$-regular.
\end{theorem}
\begin{proof}
Assume that $t \leq w-\lfloor e\sqrt{w}\rfloor$ and by \cref{prop:lastEtSc} we can assume that $t \geq e$ throughout the proof
and the associated polynomial $\sigma_e(w,\delta,m) > 0$ for all $1 \leq m \leq e$.
%(recall that if $\cC$ is $t$-regular, then it is also $t'$-regular for all $0 \leq t' < t$).
For $t$-regularity, we examine the polynomial $\sigma_e(w,\delta,t)$ and expand \cref{eq:the_key_polynomial}
to find the largest $t$ such that $\sigma_e(w,\delta,t) >0$ and also $\sigma_e(w,\delta,t') >0$ for each $t' <t$.
Using the binomial identity $\binom{t}{j}\binom{w-j}{i-j}= \frac{\binom{w}{i}}{\binom{w}{t}} \binom{i}{j} \binom{w-j}{t-j}$,
we may write
$$
\sigma_e(w,\delta,t)= \sum_{i=0}^{e} \frac{\binom{w}{i}}{\binom{w}{t}} \sum_{j=0}^{i}(-1)^j \binom{i}{j} \binom{w-j}{t-j} \binom{w+\delta-t+j}{i}.
$$
For $0\le i\le e$, we define the inner sums
$$
S_i= \sum_{j=0}^{i}(-1)^j \binom{i}{j} \binom{w-j}{t-j} \binom{w+\delta-t+j}{i}.
$$
Trivially $S_0=\binom{w}{t}>0$ and hence to prove the $t$-regularity it is sufficient to prove that $S_i \geq 0$ for all $1 \leq i \leq e$.

Denote by $[x^n]g(x)$ the coefficient of $x^n$ in $g(x)$.
Consider the coefficient extractions $\binom{w-j}{t-j}= [x^t]x^j(1+x)^{w-j}$ and $\binom{w+\delta-t+j}{i}= [y^i](1+y)^{w+\delta-t+j}$ to obtain
\begin{align*}
S_i&=[x^t y^i]\sum_{j=0}^{i}(-1)^j \binom{i}{j}x^j(1+x)^{w-j}(1+y)^{w+\delta-t+j}\\
&=[x^t y^i](1+x)^w (1+y)^{w+\delta-t} \sum_{j=0}^{i}\binom{i}{j}\left(-\frac{x(1+y)}{1+x}\right)^j .
\end{align*}
Since
$$
\sum_{j=0}^{i}\binom{i}{j}\Bigl(-\frac{x(1+y)}{1+x}\Bigr)^j =\Bigl(1-\frac{x(1+y)}{1+x}\Bigr)^i= \Bigl(\frac{(1+x)-x(1+y)}{1+x}\Bigr)^i=\Bigl(\frac{1-xy}{1+x}\Bigr)^i ,
$$
it follows that
\begin{align*}
S_i&=[x^t y^i](1+x)^w (1+y)^{w+\delta-t} \left(\frac{1-xy}{1+x}\right)^i\\
&=[x^t y^i](1+x)^{w-i} (1+y)^{w+\delta-t} (1-xy)^i .
\end{align*}
Expanding $(1-xy)^i= \sum_{k=0}^i (-1)^k\binom{i}{k}x^k y^k$, we have that
\begin{align*}
S_i&= [x^t y^i](1+x)^{w-i} (1+y)^{w+\delta-t} \sum_{k=0}^i (-1)^k \binom{i}{k}x^k y^k\\
&= \sum_{k=0}^i (-1)^k\binom{i}{k} [x^t y^i](1+x)^{w-i} (1+y)^{w+\delta-t}x^k y^k\\
&= \sum_{k=0}^i (-1)^k\binom{i}{k} \bigg([x^t]x^k(1+x)^{w-i}\bigg) \bigg([y^i]y^k(1+y)^{w+\delta-t}\bigg)\\
&= \sum_{k=0}^i (-1)^k\binom{i}{k} \binom{w-i}{t-k}\binom{w+\delta-t}{i-k}.
\end{align*}
Define
$$
A_k \triangleq \binom{i}{k}\binom{w-i}{t-k}\binom{w+\delta-t}{i-k},
$$
and hence
$$
S_i= \sum_{k=0}^i (-1)^k A_k .
$$
If $i$ is odd, then
$$
S_i = \sum_{k=0}^{(i-1)/2} (A_{2k} - A_{2k+1})
$$
and if $i$ is even, then
$$
S_i = \sum_{k=0}^{(i-2)/2} (A_{2k} - A_{2k+1}) +A_i .
$$
Since $1 \leq i \leq e$ we have that $w \geq t \geq i$. Therefore,
$A_i = \binom{i}{i}\binom{w-i}{t-i}\binom{w+\delta-t}{i-i} >0$, which implies that to prove that $S_i \geq 0$ for all $1 \leq i \leq e$,
it suffices to prove that each $A_k-A_{k+1}$ is nonnegative for even $k$.
Consider $A_k$ and $A_{k+1}$ for even $k$, and the fraction
\begin{equation}
\label{eq:basic_ration}
\frac{A_{k+1}}{A_k}= \frac{i-k}{k+1} \cdot \frac{t-k}{w-i-t+k+1} \cdot \frac{i-k}{w+\delta-t-i+k+1}.
\end{equation}
Note that $k+1 \leq i$, and $i-k$, $t-k$, $k+1$, $w-i-t+k+1$, and $w+\delta-t-i+k+1$ are positive, and hence $A_{k+1}$ and $A_k$ are positive.
Since $0 \leq k+1 \leq i \leq e$ we have that $i-k \leq e$ and $t -k \leq t$, and hence
$$
\frac{A_{k+1}}{A_k}= \frac{i-k}{k+1} \cdot \frac{t-k}{w-i-t+k+1} \cdot \frac{i-k}{w+\delta-t-i+k+1}\leq \frac{e^2 t}{(w-t-e+1)^2}.
$$
For $t \leq w-\lfloor e\sqrt{w}\rfloor$ and $e \geq 2$, we have
$$
e^2 t \leq e^2 (w-e\sqrt{w} +1)\leq e^2 (w -2\sqrt{w}+1)=(e\sqrt{w}-e)^2 \leq (w-t-e+1)^2
$$
which implies that
$$
\frac{A_{k+1}}{A_{k}} \leq \frac{e^2 t}{(w-t-e+1)^2} \leq 1 ,
$$
and hence $S_i \geq 0$ for all $1 \leq i \leq e$, which implies that
$\sigma_e(w,\delta,t) >0$ for any $t \leq w - \lfloor e\sqrt{w}\rfloor$.

Thus, $\cC$ is $\bigl(w-\lfloor e\sqrt{w}\rfloor\bigr)$-regular.
\end{proof}

The strength $\varphi$ of an $e$-perfect code in J$(n,w)$ must be less than $w$ since
not all $w$-subsets are codewords. The difference $d=w-\varphi$ is called the \emph{strength redundancy}
and it is of great importance in our exposition.

\begin{corollary}
\label{cor:str_redund}
If $\cC$ is an $e$-perfect code in \textup{J}$(n,w)$, where $e \geq 2$, then the strength redundancy~$d$
of $\cC$ satisfies $d\leq \lfloor e\sqrt{w}\rfloor$.
\end{corollary}

For the next consequence we need the following result from~\cite{EtSc04}, which will be improved later for most radii.
\begin{proposition}[Theorems~15 and~16 in \cite{EtSc04}]
\label{prop:lower bound of w}
Assume that there exists an $e$-perfect code in J$(n,w)$.
\begin{itemize}
\item If $n>2w$ and $n$ is odd, then $w>\frac{e(e+1)(e+2)}{2}+2e+1$.
\item If $n>2w$ and $n$ is even, then $w>e(e+1)(e+2)+2e+1$.
\item If $n=2w$, then $w \geq 2e^2+4e+1$.
\end{itemize}
\end{proposition}

\begin{corollary}
\label{cor:more_str_bound}
If $\cC$ is an $e$-perfect code in \textup{J}$(n,w)$, where $e \geq 2$, then the strength $\varphi$ of $\cC$ is at least $\lfloor w/2 \rfloor$
and the strength redundancy $d$ of $\cC$ satisfies $d\leq \lceil w/2 \rceil$.
\end{corollary}
\begin{proof}
In \cref{eq:basic_ration} we have $t \leq w/2$ and $0 \leq k \leq i \leq e$.
The ratio in \cref{eq:basic_ration} is at most $\frac{e^2\lfloor w/2\rfloor}{(\lceil w/2\rceil-e+1)^2}$
and it is verified when we take $\delta=0$, $t=\lfloor w/2 \rfloor$, $k=0$, and $i=e$ in \cref{eq:basic_ration}.
Therefore, the sum of each even-odd pair is nonnegative if
$$
e^2\lfloor w/2\rfloor\le (\lceil w/2\rceil-e+1)^2.
$$
For $w\ge 2e^2+4e$, the above inequality holds and it is obtained for each $e$-perfect code due to \cref{prop:lower bound of w}.
Applying the same steps as in the proof of \cref{thm:regular for all w} yields strength $\varphi \geq \lfloor w/2 \rfloor$
and strength redundancy $d\leq \lceil w/2 \rceil$.
\end{proof}

By combining \cref{prop:p2_divides_ball} with \cref{cor:more_str_bound} we have the following theorem.
\begin{theorem}
\label{thm:sufficient condition}
If $e \geq 2$ and $\Phi_e(n,w)\equiv 0~(\mmod~{p^2})$
for some prime $p$, then no $e$-perfect code exists in \textup{J}$(n,w)$.
\end{theorem}

\cref{prop:no-perfect-3} together with \cref{thm:sufficient condition} yields the first important nonexistence
result for $e$-perfect codes, reducing considerably the range of radii in which $e$-perfect codes can exist.
\begin{corollary}
\label{cor:square-factor-e-plus-one}
If $p^2$ divides $e+1$ for some prime $p$, then no $e$-perfect code exists in the Johnson scheme.
In particular there is no $e$-perfect code in the Johnson scheme when $e \equiv -1~(\mmod~4)$ or $e \equiv -1~(\mmod~9)$.
\end{corollary}

\section{Prime Factorization for the Size of a Ball}
\label{sec:divide_ball}

In this section we consider primes that divide the size of a ball $\Phi_e(n,w)$ if there exists an $e$-perfect code in J$(n,w)$.
It is shown that in a short interval of length $e/3$ starting with $e+1$ each prime divides $\Phi_e (n,w)$.
We continue and find a small interval of length $d$ in which each prime that divides $\Phi_e(n,w)$ has a power inside the interval.
This will lead to several consequences for excluding possible radii with perfect codes in the other sections.

\begin{lemma}
\label{lem:block-congruence-mod-p}
Let $p$ be a prime, let $P=p^k$, where $k\ge 1$, and let
$$
e=P\epsilon+e_0,\qquad w=P\alpha+b,\qquad \nminusw=P\beta+c,
$$
where $0\le e_0<P$ and $0\le b,c\le e_0$. If $p\mid\binom{b+c}{b}$, then $p$ divides $\Phi_e(n,w)$.
\end{lemma}
\begin{proof}
By Lucas' theorem, for $0\le r<P$, we have that
$$
\binom{w=P\alpha+b}{Pm+r}\equiv \binom{\alpha}{m}\binom{b}{r}~(\mmod~p)~~ \text{and} ~~ \binom{\nminusw=P\beta+c}{Pm+r}\equiv\binom{\beta}{m}\binom{c}{r}~(\mmod~p).
$$
Since $0 \leq b,c \leq e_0 <P$, it follows that
$$
\sum_{r=0}^{P-1}\binom{b}{r}\binom{c}{r}=\sum_{r=0}^{e_0}\binom {b}{b-r}\binom{c}{r}=\binom{b+c}{b}.
$$
By splitting the sum of $\Phi_e(n,w)$ in \cref{{eq:ball_size}} into $\epsilon$ blocks of size $P$,
one block of size $e_0 +1$, and computing modulo $p$, we use Lucas' theorem to obtain
\begin{align*}
\Phi_e(n,w)=\sum_{i=0}^e \binom{w}{i} \binom{\nminusw}{i} &= \sum_{m=0}^{\epsilon-1} \sum_{r=0}^{P-1} \binom{w}{Pm+r}\binom{\nminusw}{Pm+r}
+\sum_{r=0}^{e_0} \binom{w}{P\epsilon+r}\binom{\nminusw}{P \epsilon+r}\\
&
\equiv \sum_{m=0}^{\epsilon-1} \sum_{r=0}^{P-1} \binom{\alpha}{m}\binom{b}{r}\binom{\beta}{m}\binom{c}{r}+\sum_{r=0}^{e_0} \binom{\alpha}{\epsilon}\binom{b}{r}\binom{\beta}{\epsilon}\binom{c}{r}\\
&
\equiv \sum_{m=0}^{\epsilon-1}  \binom{\alpha}{m}\binom{\beta}{m} \sum_{r=0}^{P-1} \binom{b}{r}\binom{c}{r}+\binom{\alpha}{\epsilon}\binom{\beta}{\epsilon}\sum_{r=0}^{e_0} \binom{b}{r}\binom{c}{r}\\
&
\equiv  \binom{b+c}{b}\sum_{m=0}^{\epsilon-1}\binom{\alpha}{m}\binom{\beta}{m}+ \binom{\alpha}{\epsilon}\binom{\beta}{\epsilon}\binom{b+c}{b} \equiv0~(\mmod~p).
\end{align*}
\end{proof}

Let $\nu_p (\mu)$ denote the $p$-adic valuation of $\mu$, i.e., $\nu_p (\mu) \triangleq \max \{ i ~:~ p^i ~ \text{divides} ~ \mu \}$.

\begin{lemma}
\label{lem:residue-localization}
Let $p$ be a prime, let $P=p^k$, where $k\ge 1$, and let $b$ and $c$ satisfy $b \equiv w~(\mmod~P)$, $c \equiv \nminusw~(\mmod~P)$, and $0 \leq b,c <P$. Suppose that
$$
e\equiv -\gamma \pmod P \qquad \text{and} \qquad 1\le \gamma\le \frac{P+1}{2}.
$$
If there exists an $e$-perfect code in \textup{J}$(n,w)$, then $P- 2\gamma+1\le b,c\le P-\gamma$.
\end{lemma}
\begin{proof}
Let $\cC$ be an $e$-perfect code in J$(n,w)$. By \cref{prop:embedded_Steiner} there
exist a Steiner system S$(e+1,2e+1,w)$ and a Steiner system S$(e+1,2e+1,\nminusw)$.
Since $1\le \gamma\le \frac{P+1}{2}$ we have that $-\gamma$ is restricted to the range between $\frac{P-1}{2}$ and $P-1$ modulo $P$
and hence $e$ is congruent to the range between $\frac{P-1}{2}$ and $P-1$ modulo $P$ which implies that $e+1 \geq \gamma$.
By \cref{prop:necessary_Steiner} this implies that
\begin{equation}
\label{eq:binomial_necc_embed}
\binom{e+\gamma}{\gamma}\bigg| \binom{w-e+\gamma-1}{\gamma}.
\end{equation}
If $m=\nu_p(\gamma)$, then since $e+\gamma \equiv 0~(\mmod~ {p^k})$, i.e., $\nu_p (e+\gamma) \geq k$,
it follows by Kummer's theorem that
$$
\nu_p\bigl(\binom{e+\gamma}{\gamma}\bigr)\geq k-m.
$$
Therefore, by \cref{eq:binomial_necc_embed} we have that
\begin{align}
\label{eq:nu_p ge k-s}
\nu_p\bigl(\binom{w-e+\gamma-1}{\gamma} \bigr) \geq \nu_p\bigl(\binom{e+\gamma}{\gamma}\bigr) \geq k-m.
\end{align}

We claim that for any $0\leq x<P$, if $x\equiv w-e~(\mmod~P)$, we have
\begin{equation}
\label{eq:values_x}
x=0\quad\text{or}\quad x\geq P-\gamma+1.
\end{equation}
Since $x\equiv w-e~(\mmod~P)$, it follows that $w-e+\gamma-1 \equiv x+\gamma-1~(\mmod~P)$ and hence
if $1 \leq x \leq P-\gamma$, then $x+\gamma -1 <P$ and the least significant $k$ digits in base-$p$ representing $w-e+\gamma-1$
form the base-$p$ representation of $x+\gamma -1$.
This implies that when $\gamma$ is added to $w-e-1$
there is no carry at the place representing
$p^{k-1}$ since such a carry implies that $x+\gamma-1$ is at least~$P$, a contradiction.
Hence, there is also no carry in places representing $p^i$ for $i \geq k$.
Since $m=\nu_p(\gamma)$, it follows that carries are possible
in places representing $p^j$, $j \geq m$, and hence there are at most $k-1-m$ carries, contradicting \cref{eq:nu_p ge k-s},
and the claim of \cref{eq:values_x} is proved.

Since $w\equiv e+x\equiv x-\gamma ~(\mmod~P)$, we obtain
$$
b \equiv w~(\mmod~P)\equiv~x-\gamma~(\mmod~P), ~~~ b \in R_{\gamma}=\{P-2\gamma+1,\ldots,P-\gamma\}.
$$
(note that $b=P-\gamma$ occurs when $x=0$).
The same argument can be applied to $\nminusw$ to obtain $c\equiv \nminusw~(\mmod~P)$ and $c \in R_{\gamma}$.
\end{proof}

\begin{lemma}
\label{lem:modulo-p-criterion}
Let $p$ be a prime, let $P=p^k$ for some $k \geq 1$, and suppose that
$$
e\equiv-\gamma~(\mmod~P) \qquad \text{and}  \qquad 1 \leq \gamma\le\left\lfloor\frac{P+2}{4}\right\rfloor.
$$
If an $e$-perfect code exists in \textup{J}$(n,w)$, then $p$ divides $\Phi_e(n,w)$.
\end{lemma}
\begin{proof}
Let $b \equiv w~(\mmod~P)$, $c \equiv \nminusw~(\mmod~P)$, where $0 \leq b,c <P$.
By~\cref{lem:residue-localization} we have that
$$
P-2\gamma+1\le b,c\le P-\gamma.
$$
This implies that $b+c \geq 2P-4\gamma +2$ and since $\gamma\le\left\lfloor (P+2)/4 \right\rfloor$ we have that $4\gamma \leq P+2$
and hence $b+c \geq P$. Therefore, the base-$p$ addition of $b$ and $c$ has at least one carry from the $p^{k-1}$-place.
Kummer's theorem implies now that $p$ divides $\binom{b+c}{b}$, and hence by \cref{lem:block-congruence-mod-p} we have that
$p$ divides $\Phi_e(n,w)$.
\end{proof}

\begin{corollary}%[Moving detected primes]
\label{cor:moving-prime-detector}
If an $e$-perfect code exists in \textup{J}$(n,w)$ and $p$ is a prime satisfying $e<p\le(4e+2)/3$, then $p$ divides $\Phi_e(n,w)$.
\end{corollary}
\begin{proof}
Take $P=p$ and $\gamma=p-e$ in \cref{lem:modulo-p-criterion}.
\end{proof}

In view of \cref{cor:moving-prime-detector} we have that every prime in the set
$$
\cP_e \triangleq \left\{p\text{ prime}:e<p\leq  \frac{4e+2}{3}\right\}
$$
divides $\Phi_e(n,w)$.

We continue to consider primes that might divide $\Phi_e(n,w)$.
The strength of an $e$-perfect code in J$(n,w)$ is important to improve the divisibility condition
implied by \cref{eq:code_size}. This improvement was given in~\cite{EtSc04}.

\begin{proposition} [Theorem 17 in~\cite{EtSc04}]
\label{prop:divisiility_strength}
Let $e \geq 1$ and $\cC$ be an $e$-perfect code in \textup{J}$(n,w)$ with strength $\varphi$. For each $0\le i\le \varphi$, the code is an $i$-design,
and the number of codewords containing a fixed $i$-subset is
\begin{equation}
\label{eq:lambda_i}
\frac{\abs{\cC}\binom{w}{i}}{\binom{n}{i}} =\frac{1}{\Phi_e(n,w)}\binom{n-i}{w-i}.
\end{equation}
Consequently,
$$
\Phi_e(n,w) \big| \binom{n-i}{w-i},\qquad 0\le i\le \varphi.
$$
\end{proposition}

The following refinement of \cref{prop:divisiility_strength}
simultaneously prescribes coordinates that must occur in a codeword and coordinates that must be avoided.
The following theorem is the key in finding the prime factorization of $\Phi_e(n,w)$.
\begin{theorem}
\label{thm:mixed-design-divisibility}
Let $e \geq 1$ and $\cC$ be an $e$-perfect code in $\textup{J}(n,w)$, on the ground set $[n]$, with strength $\varphi$.
Let $X,Y\subseteq [n]$ be two disjoint subsets,
with
$$
|X|=r,\qquad|Y|=m,\qquad r+m\le \varphi.
$$
Then, the number of codewords $Z\in\cC$ satisfying $X\subseteq Z$ and $Z\cap Y=\varnothing$ is $\frac{1}{\Phi_e(n,w)}\binom{n-r-m}{w-r}$
and hence,
$$
\Phi_e(n,w) \Big| \binom{n-r-m}{w-r}.
$$
\end{theorem}
\begin{proof}
Let
$$
\lambda(X,Y)=\big|\{Z\in \cC: X\subseteq Z,\ Z\cap Y=\varnothing \} \big|.
$$
For each $y \in Y$, let $A_y = \{ Z \in \cC ~:~ X \cup \{ y \} \subseteq Z \}$.
Using the inclusion--exclusion principle we have that
$$
\lambda(X,Y)=\sum_{U \subseteq Y}(-1)^{|U|}\big|\{Z\in \cC ~:~ X\cup U \subseteq Z\}\big|.
$$
Since $|X\cup U|\leq r+m\le \varphi$, it follows by \cref{eq:lambda_i} that
$$
\big|\{Z\in \cC ~:~ X\cup U\subseteq Z\}\big| =\binom{n-r-|U|}{w-r-|U|} \Big/ \Phi_e(n,w).
$$
Therefore, we have that
$$
\lambda(X,Y)=\frac{1}{\Phi_e(n,w)}\sum_{u=0}^m(-1)^u \binom{m}{u}\binom{n-r-u}{w-r-u}.
$$
The sum $\sum_{u=0}^m(-1)^u \binom{m}{u}\binom{n-r-u}{w-r-u}$ is the enumeration by inclusion--exclusion
of all $w$-subsets of $[n]$ which contain $X$ and avoid $Y$. Hence,
$$
\sum_{u=0}^m(-1)^u\binom{m}{u}\binom{n-r-u}{w-r-u}=\binom{n-r-m}{w-r}.
$$
Therefore,
$$
\lambda(X,Y)=\frac{1}{\Phi_e(n,w)}\binom{n-r-m}{w-r}.
$$
Since $\lambda(X,Y)$ is an integer, it follows that $\Phi_e(n,w)$ divides $\binom{n-r-m}{w-r}$.
\end{proof}

%If $r+m=\varphi$ in \cref{thm:mixed-design-divisibility}, then we have divisibility
%conditions for a closed interval of binomial coefficients in a single row of the Pascal's triangle.

\begin{corollary}
\label{cor:interval_overlap}
Assume that $e \geq 1$ and $\cC$ is an $e$-perfect code with strength redundancy $d$ in \textup{J}$(n,w)$. Then,
\begin{equation}
\label{eq:gcd_binom}
\Phi_e(n,w) \big| \hspace{-0.1cm} \gcd_{d \le k\le \nminusw} \hspace{-0.1cm} \bigl( \binom{\nminusw+d}{k} \bigr).
\end{equation}
\end{corollary}
\begin{proof}
Applying \cref{thm:mixed-design-divisibility} with $r=w-d-j$ and $m=j$ implies that $r+m=w-d$, and therefore,
$$
\Phi_e(n,w) \big|  \binom{\nminusw+d}{d+j}.
$$
Letting $j$ run from $0$ to $w-d$ implies that
\begin{equation}
\label{eq:gcd_binom1}
\Phi_e(n,w) \big| \hspace{-0.1cm}\gcd_{d\le k\le w} \bigl( \binom{\nminusw+d}{k} \bigr).
\end{equation}
Let $M=\nminusw+d$. \cref{eq:gcd_binom1}, together with the binomial identity $\binom{M}{k}=\binom{M}{M-k}$, implies that
\begin{equation}
\label{eq:gcd_binom2}
\Phi_e(n,w) \big|  \hspace{-0.2cm} \gcd_{\delta +d \le k\le \nminusw} \hspace{-0.2cm} \bigl( \binom{\nminusw+d}{k} \bigr).
\end{equation}
By \cref{eq:Roos} we have that $e \geq 2$ implies that $\delta \leq \lfloor w/2 \rfloor$, and by \cref{cor:more_str_bound} we have that
$d \leq \lceil w/2 \rceil$. This implies that $\delta+d \leq w$ and hence the intervals in \cref{eq:gcd_binom1,eq:gcd_binom2}
overlap, and therefore we obtain \cref{eq:gcd_binom}.

For $e = 1$, by \cref{prop:Martin} we have that $d=w-\varphi=k+1$ and $\delta=r-k-1$ and hence
$\delta + d = r < w=kr+1$ which implies that the intervals in \cref{eq:gcd_binom1,eq:gcd_binom2} overlap and as a consequence
we obtain \cref{eq:gcd_binom}.
\end{proof}

The next lemma evaluates exactly the gcd in \cref{cor:interval_overlap}.
Besides proving that this gcd is squarefree, it identifies its prime
divisors through prime powers lying in a specific given short interval.
Moreover, by \cref{cor:interval_overlap} the prime divisors of $\Phi_e(n,w)$
must divide this gcd and with no larger multiplicity for each prime divisor.
This will be used later to exclude the remaining radii for the existence of $e$-perfect codes in J$(n,w)$.
For consistency with later sections we will assume that $M=\nminusw +d$, although some of the results are more general.
Similarly, we retain for later sections the notation $s_p$ as used in the next lemma.
This lemma can be proved as a consequence of Theorems 1 and 2 in~\cite{JOS85} on the gcd
of consecutive binomial coefficients; we include a short self-contained proof for completeness.

\begin{lemma}
\label{lem:central-row-gcd}
If $M$ and $\ell$ are positive integers such that $M \ge 3\ell-1$, then
$$
\gcd_{\ell \le k\le M-\ell} \bigl( \binom{M}{k} \bigr)
= \hspace{-0.5cm}
\prod_{\substack{p\text{ prime} \\ M-\ell<p^{s_p}\le M\\ \text{for some }s_p\ge 1}} \hspace{-0.7cm} p.
$$
In particular, this gcd is squarefree.
\end{lemma}
\begin{proof}
Let $p\leq M$ be a prime and $Q$ be the largest power of $p$ not exceeding~$M$, i.e., $M \geq Q$.
It is easy to verify that there exists some $s_p\geq 1$ satisfying $ M-\ell<p^{s_p}\le M$ if and only if $Q> M-\ell$.
To prove the claim of the lemma we distinguish between two cases depending on whether $Q\le M-\ell$ or $Q> M-\ell$.

\begin{enumerate}
\item First, suppose that $Q\le M-\ell$. We have to show that $p$ does not divide the gcd,
that is $p$ does not divide $\binom{M}{r}$ for some $\ell \le r\le M-\ell$.
If $Q\ge \ell$ then let $r=Q$ and hence $\ell \le r\le M-\ell$. If $Q<\ell$ then let
$r=\left\lceil \ell /Q\right\rceil \cdot Q$ and therefore $\ell \le r \le 2\ell-1\le M-\ell$.
Therefore, in both cases we have
$\ell \leq r\le M-\ell$ and $r=mQ$ for some $1 \leq m \leq p-1$.

Let $M=aQ+\rho$, where $0\le\rho<Q$ and $1\le a\le p-1$.
Since $r=mQ<M$, its coefficient~$m$ in the $Q$-place is at most $a$. Hence, $M-r=aQ+\rho -mQ= (a-m)Q+\rho$ which
implies that there is no carry when $r$ is added to $M-r$. Therefore,
Kummer's theorem implies that
$$
p\nmid \binom{M}{r}.
$$
Thus, $p$ does not divide $\gcd_{\ell\le k\le M-\ell}(\binom{M}{k})$.

\item Next, suppose that $Q>M-\ell$. We claim that $p$ divides $\binom{M}{k}$ for
all $\ell \le k\le M-\ell$. Since $M \geq Q$ we also have
$$
M=Q+\rho,  \qquad  0\le\rho<\ell.
$$

By our definition and assumption on $\rho$ and $Q$, we have, for any integer $k$ in $[\ell,M-\ell]$,
that $\rho< k< Q$ and $0 < M-k <Q$. Since $Q$ is a power of $p$ not exceeding~$M$, and $M-k <Q$,
it follows that adding $k$ to $M-k$ has at least one carry. Assume $Q$ has in base-$p$ its unique~$1$ at position $\tau$.
Then, there is a carry in the addition of $k$ to $M-k$ at position $\tau -1$.
Thus, by Kummer's theorem, every $\binom{M}{k}$ with $k$ in the interval $[\ell,M-\ell]$ is divisible by $p$
and the claim is proved.
\end{enumerate}

It remains to show that when $Q>M-\ell$, the minimum $p$-adic valuation in $\{\nu_{p} ( \binom{M}{k}) ~:~k\in [\ell,M-\ell]\} $ is one,
that is, there exists a $k\in [\ell,M-\ell]$ with $\nu_{p} ( \binom{M}{k})=1$.
Since $M\ge 3\ell-1$ and $Q > M-\ell$, we have that $Q\ge 2\ell$. Denote $B=Q/p$.
If $\ell \le B$,  let $k=B$. This implies that $\ell \le k=B=Q/p\le Q-B\le M-\ell$.
If $\ell>B$, let $k=\left\lceil \ell/B\right\rceil \cdot B$. Then $\ell \le k=\left\lceil \ell/B\right\rceil \cdot B\le 2\ell-1\le M-\ell$.
Therefore, in both cases, $\ell \le k\le M-\ell$ and $k=mB$ for some $1\le m\le p-1$.

Since $\rho<\ell \le k=mB$, we can analyse the number of carries when we add $k$ to $M-k$ as follows:
$Q$ has in base-$p$ exactly one~$1$ at position $\tau$ and hence $B=Q/p$ has exactly one~$1$ at position $\tau -1$
and $k=mB$ has exactly one nonzero entry at position $\tau -1$, where the value there is $m$.
When $k=mB$ is added to $M-k<Q$ we have a carry from position $\tau -1$ since the most significant digit of $k$ and $M-k$ is
less than $\tau$ and the most significant digit of $M \geq Q$ is~$\tau$. Since $k=mB$ has an $m$ only at position $\tau -1$,
it follows that this is the only carry when $k$ is added to $M-k$. Hence, by Kummer's theorem we have
$$
\nu_{p} \bigl( \binom{M}{k} \bigr) =1,
$$
which completes the proof of the lemma.
\end{proof}

Note that \cref{lem:central-row-gcd} is a general result, regardless of whether there exists an $e$-perfect code in J$(n,w)$.
If there exists an $e$-perfect code in J$(n,w)$, then
apply \cref{lem:central-row-gcd} with $M=\nminusw+d$ and $\ell=d$. Since for $e \geq 2$, $\nminusw \ge w$ and $d\le\lceil w/2\rceil$
by Corollary~\ref{cor:more_str_bound}, one has $M\ge 3d-1$.
For $e=1$, by \cref{prop:Martin} we have $w=kr+1$ and $r \geq d=k+1$ which implies that $w \geq 2d-1$
and hence $M=\nminusw+d \geq 3d-1$.
Therefore, we have the following consequence of \emph{localization}.

\begin{corollary}
\label{cor:prime factors from chain divisitilities}
If there exists an $e$-perfect code in $\textup{J}(n,w)$, then
$$
\Phi_e(n,w) \big| \hspace{-0.6cm} \prod_{\substack{p\text{ prime}\\ \nminusw<p^{s_p}\le \nminusw+d\\ \text{ for some }s_p\ge 1}} \hspace{-0.6cm}p.
$$
In particular, \(\Phi_e(n,w)\) is squarefree.
\end{corollary}

Note that $s_p$ is unique for each prime $p$,
since $M < 2 \nminusw$ (by \cref{cor:more_str_bound} for $e \geq 2$ and by the parameters in \cref{prop:Martin} for $e=1$)
implies that two consecutive powers of the same prime
cannot be contained in $(\nminusw,M]$.
It is worth mentioning that \cref{cor:prime factors from chain divisitilities} also implies the result of \cref{thm:sufficient condition}
that $\Phi_e(n,w)$ is squarefree. Moreover, unlike \cref{thm:sufficient condition} which applies to $e \geq 2$,
\cref{cor:prime factors from chain divisitilities}
holds for $e \geq 1$. Finally, the interval $(\nminusw,M]$ will be called the \emph{localized interval} and the exponent $s_p$ will
be called the \emph{localized exponent} of $p$.

\section{The Algebra of the Johnson Scheme}
\label{sec:JohnsonAlgebra}

In this section we present a short introduction to association schemes with an emphasis on the Johnson scheme.
For completeness we present more than what is required for our exposition.
We start with some of the theory given by Biggs~\cite[pp. 291--292]{Big73}.

Let $A_i$, $0 \leq i \leq w$, be the distance-$i$ adjacency matrix,  with columns and rows indexed by $\nbinw$,
where $A_0$ is the $\binom{n}{w} \times \binom{n}{w}$ identity matrix, $A=A_1$, and
$$
(A_i)_{\bldx,\bldy} =
\begin{cases}
    1  & \text{if} ~ d(\bldx,\bldy)=i, \\
    0 & \text{otherwise}.
\end{cases}
$$

The preceding definitions yield a recurrence with the distance-$i$ adjacency matrices
\begin{equation}
\label{eq:section5-matrix-three-term-recurrence}
A \cdot A_i=b_{i-1}A_{i-1}+a_iA_i+c_{i+1}A_{i+1}, \qquad 1\le i<w,
\end{equation}
where $a_i =i(n-2i)$, $b_i=(w-i)(\nminusw-i)$, and $c_i = i^2$ for each $0 \leq i \leq w$.
This is the tridiagonal multiplication rule for the distance matrices (see Biggs~\cite[p. 292]{Big73}),
where the \emph{eigenvector sequence} of polynomials was defined as follows.
$$
v_0(x) \triangleq 1,\qquad v_1(x) \triangleq x.
$$
The matrix identity
\begin{equation}
\label{eq:section5-distance-polynomial-matrices}
A_i=v_i(A), \qquad 0\le i\le w
\end{equation}
was proved by Biggs~\cite[p.~292]{Big73}.
Comparing the leading coefficients in \cref{eq:section5-matrix-three-term-recurrence}, gives
\begin{equation}
\label{eq:section5-vi-degree-leading}
\deg v_i(x)=i,\qquad \operatorname{leading~coefficient}~(v_i (x))=\frac1{(i!)^2}, \qquad 0\le i\le w.
\end{equation}

The Bose--Mesner~\cite{BoMe59} algebra of the Johnson scheme is the real algebra
spanned by $A_0,\ldots,A_w$. The matrices $A_i$ are real symmetric and commute and hence
this algebra has a unique set of mutually orthogonal primitive idempotents $J_0,\ldots,J_w$;
equivalently, the $J_i$s are the orthogonal projections onto the common eigenspaces of the matrices $A_j$s. They satisfy
$$
J_i^2=J_i,\qquad J_k J_i=0, ~~ \text{where} ~~ k\neq i, \qquad \sum_{i=0}^w J_i=I,\qquad J_0=\frac{1}{\binom{n}{w}}{J},
$$
where $J$ is the all-one matrix (not to be confused with the $J_i$s).
We define $p_i(j)$ to be the eigenvalue of $A_i$ with the idempotent $J_j$, i.e.,
\begin{equation*}
A_i J_j=p_i(j) J_j,\qquad A_i=\sum_{j=0}^w p_i(j)J_j, ~~ 0 \leq i,j \leq w .
\end{equation*}
Thus $P=(p_i(j))$ is the first eigenmatrix. Since $A_i=v_i(A)$ we have that
$$
p_i(j)=v_i(\theta_j),\qquad \theta_j \triangleq p_1(j).
$$
The second eigenmatrix is defined by the eigenvalues $q_j(i)$ as follows
\begin{equation}
\label{eq:section5-second-eigenmatrix-convention}
J_j=\frac1{\binom{n}{w}}\sum_{i=0}^w q_j(i)A_i.
\end{equation}
The matrix $Q=(q_j(i))$ with these entries is the second eigenmatrix.

For the Johnson scheme the following polynomial, called an Eberlein polynomial, is defined in the following proposition.
\begin{proposition}[Theorem~4.6 and Eq.~(4.35) in \cite{Del73}, Theorem 10, p. 665 in~\cite{McSl77}]
\label{prop:quoted-delsarte-eigenmatrix}
The Eberlein polynomials for \textup{J}$(n,w)$ are defined by
$$
E_k(x)=\sum_{j=0}^{k}(-1)^j\binom{x}{j}\binom{w-x}{k-j}\binom{\nminusw-x}{k-j}.
$$
The associated eigenvalues of the related eigenmatrices $P$ and $Q$ of the scheme
(defined by their eigenvalues $p_k(i)$ and $q_k(i)$ as given in~\cite[p. 654]{McSl77} and in the sequel) are:
$$
p_k(i)=E_k(i),\qquad q_k(i)=\mu_k k_i^{-1}E_i(k),\qquad 0\le i,k\le w,
$$
where $k_i =\binom{w}{i} \binom{\nminusw}{i}$ is the number of vertices ($w$-subsets) at distance $i$ from a given vertex and
$\mu_i = \frac{n-2i+1}{n-i+1} \binom{n}{i}$ is the rank of $J_i$ and the multiplicity of the eigenvalue $\theta_i$ for $A_1$.
\end{proposition}

The explicit Johnson eigenmatrix formula gives
\begin{equation}
\label{eq:section5-Johnson-adjacency-eigenvalues}
\theta_j=(w-j)(\nminusw-j)-j=w \nminusw-j(n+1-j),\qquad 0\le j\le w.
\end{equation}
These eigenvalues are distinct: for $0\le j<w$, and since $n \geq 2w$ we have
$$
\theta_j-\theta_{j+1}=n-2j>0.
$$

Let $\chi\in\{0,1\}^{\nbinw}$ be the characteristic column vector of an $e$-perfect code $\cC$ in J$(n,w)$, and define
$$
\cD_i \triangleq \frac1{|\cC|}\chi^{T}A_i\chi,
\qquad
\cD_j^* \triangleq  \sum_{i=0}^w\cD_i q_j(i).
$$
Delsarte's inner distribution $\cD_i$ is the average number of codewords at distance $i$ from a codeword.
With this convention in
\cref{eq:section5-second-eigenmatrix-convention}, $\cD_j^*$ is the $j$th coordinate
of the dual distribution. Moreover,
\begin{equation}
\label{eq:section5-dual-distribution-projection}
\cD_j^*=\frac{\binom{n}{w}}{|\cC|}\chi^T J_j\chi \geq 0 .
\end{equation}
Hence
$$
\cS \triangleq \{j\in\{1,\ldots,w\}:\cD_j^*>0\}  %=\{j\in\{1,\ldots,w\}:E_j\chi\ne0\}
$$
is precisely the nonprincipal spectral support of the code.
These definitions and the equivalence in \cref{eq:section5-dual-distribution-projection}
are exactly the specialization of the criterion of the following proposition.

\begin{proposition} [Equation (3.27), Theorems~3.10,~4.7 in \cite{Del73}]
\label{prop:quoted-delsarte-design}
For $T\subseteq\{1,\ldots,w\}$,
a nonempty subset $\cC$ in \textup{J}$(n,w)$ is called a $T$-design when its inner distribution $\cD=(\cD_0,\cD_1,\ldots,\cD_w)$ satisfies
$$
\sum_i \cD_iq_j(i)=0\qquad\text{for every }j\in T.
$$
If $J_0,J_1,\ldots,J_w$ are the minimal idempotents of the Bose--Mesner algebra and $\chi_\cC$ is the characteristic vector of $\cC$,
then $\cC$ is a \(T\)-design if and only if
$$
J_j\chi_\cC=0\qquad\text{for every }j\in T.
$$
If $T=\{1,2,\ldots,t\}$ with $1\leq t \leq w$, a code $\cC$ is a $T$-design
in $\textup{J}(n,w)$ if and only if it forms a $t$-design $S_\lambda(t,w,n)$ for some $\lambda$.
\end{proposition}
The statements collected in \cref{prop:quoted-delsarte-design} give, in the present notation,
$$
\cC\text{ is a }t\text{-design}\quad\Longleftrightarrow\quad
\cD_1^*=\cdots=\cD_t^*=0.
$$
Consequently, if the strength of $\cC$ is exactly $\varphi=w-d$, then
\begin{equation}\label{eq:section5-exact-strength-dual-support}
\cD_j^*=0, \quad 1\le j\le w-d, \qquad \cD_{w-d+1}^*>0.
\end{equation}
The strict inequality in the second relation follows from the nonnegativity in \cref{eq:section5-dual-distribution-projection}.

We use the radius-$e$ sum polynomial, called the Lloyd polynomial~\cite{Llo57}, originally defined for the Hamming scheme
(when Krawtchouk polynomials~\cite[p. 657]{McSl77} are used instead of the Eberlein polynomials):
\begin{equation}
\label{eq:section5-Lloyd-definition}
L_e(x) \triangleq \sum_{i=0}^e v_i(x).
\end{equation}

The \emph{external distance} (which coincides with the covering radius for a perfect code) of a code, denoted by $r_\text{ext}$ is defined as
$$
r_{\text{ext}} \triangleq |\cS|.
$$
The following result was given
originally by Lloyd~\cite{Llo57} for the Hamming scheme and it was generalized as follows to the Johnson scheme.
\begin{proposition}
[Section 5.2.2 and Theorem~5.7 in \cite{Del73}]
\label{prop:quoted-delsarte-Lloyd}
If $\cC$ is an $e$-perfect code in \textup{J}$(n,w)$, then its external distance
is $r_\text{ext}=e$. If $\cD=(\cD_0,\cD_1,\ldots,\cD_w)$ is the distance distribution of $\cC$, then the sum polynomial
$$
L_e(x)=v_0(x)+v_1(x)+\cdots+v_e(x)
$$
vanishes at $e$ distinct points, $L_e(\theta_j)=0$ for every $j \in \cS$.
\end{proposition}

As a consequences from \cref{prop:quoted-delsarte-Lloyd} we also have
\begin{equation}
\label{eq:external_support}
r_\text{ext} = |\cS| =e.
\end{equation}

The following result is an adaptation of Lloyd's theorem for the Johnson scheme.
\begin{theorem}
\label{thm:section5-Lloyd-factorization}
Assume that there exists an $e$-perfect code in \textup{J}$(n,w)$.
The polynomial $L_e (x)$ in \cref{eq:section5-Lloyd-definition} satisfies
$$
L_e(w \nminusw)=\sum_{i=0}^e\binom{w}{i}\binom{\nminusw}{i}=\Phi_e(n,w).
$$
Moreover, there are $e$ distinct integers $0\le h_1<\cdots<h_e=d-1$ such that, with
$$
\lambda_i=h_i(h_i+\delta+1)-w,
$$
one has
\begin{equation}
\label{eq:section5-Lloyd-factorization}
(e!)^2L_e(x)=\prod_{i=1}^e(x-\lambda_i),\qquad (e!)^2\Phi_e(n,w)=\prod_{i=1}^e(w-h_i)(\nminusw+h_i+1).
\end{equation}
\end{theorem}
\begin{proof}
We first evaluate $L_e(x)$ at $x=w \nminusw$. For every $i$, $0 \leq i \leq w$,
let $k_i$ be the number of vertices at distance $i$ from a fixed vertex.
Since $A_i\mathbf1=k_i\mathbf1$, while $A_i=v_i(A)$, where $0 \leq i \leq w$, we have that
$$
v_i(w \nminusw)=k_i,
$$
and therefore
$$
L_e(w \nminusw)=\sum_{i=0}^ev_i(w \nminusw)=\sum_{i=0}^e\binom{w}{i}\binom{\nminusw}{i}=\Phi_e(n,w).
$$
By \cref{eq:external_support}, the spectral support $\cS$ has exactly $e$ elements. Let
$$
\cS=\{j_1<j_2<\cdots<j_e\}.
$$
For each $j_i \in\cS$, by \cref{prop:quoted-delsarte-Lloyd} we have \(L_e(\theta_{j_i})=0\).
The numbers $\theta_0,\ldots,\theta_w$ are distinct by
\cref{eq:section5-Johnson-adjacency-eigenvalues}, and hence these are $e$ distinct roots of $L_e(x)$. Since $L_e(x)$ has degree $e$ and leading
coefficient $(e!)^{-2}$, it follows that
\begin{equation}
\label{eq:section5-factorization-by-support}
(e!)^2L_e(x)=\prod_{i=1}^e(x-\theta_{j_i}).
\end{equation}
The preceding spectral-support argument also identifies which eigenvalues occur.

We now use the exact strength of the code. By \cref{eq:section5-exact-strength-dual-support}, we have that,
$$
j_1=w-d+1 \qquad \text{and} \qquad  w-d+1 \leq j_i \leq w, \qquad 2 \leq i \leq e.
$$
Define
$$
h_i \triangleq w-j_{e+1-i}, \qquad 1\le i\le e.
$$
Since $j_1<\cdots<j_e\le w$, it follows that these are $e$ distinct integers satisfying
$$
0\le h_1<h_2<\cdots<h_e=w-j_1=d-1.
$$
Since $\nminusw=w+\delta$, it follows by \cref{eq:section5-Johnson-adjacency-eigenvalues} that
$$
\theta_{w-h_i}=(w-(w-h_i))(\nminusw-(w-h_i))-(w-h_i)=h_i(h_i+\delta+1)-w.
$$
Thus, after reversing the order of the factors in \cref{eq:section5-factorization-by-support}, the roots can be written as
$$
\lambda_i \triangleq h_i(h_i+\delta+1)-w,
$$
and we obtain the first factorization in \cref{eq:section5-Lloyd-factorization}.

Finally, evaluate that factorization at $x=w \nminusw$, which implies that for every $i$,
$$
w \nminusw-\lambda_i=w \nminusw +w-h_i(h_i+\delta+1)=(w-h_i)(\nminusw+h_i+1),
$$
where the last equality is due to $\nminusw=w+\delta$. Together with the already proved identity $L_e(w \nminusw)=\Phi_e(n,w)$, this yields
\cref{eq:section5-Lloyd-factorization} as required.
\end{proof}

\section{Complete Exclusion with Two Exceptions}
\label{sec:complete-exclusion}

In this section we start to combine all the results that were presented so far with two more
important concepts from number theory, associated with the distribution of primes,
to prove the nonexistence of $e$-perfect codes in the Johnson scheme,
unless $e=1$ or $e=2$.

For each prime $p$ which divides $\Phi_e(n,w)$ and its exponent $s_p$, define the \emph{sum of reciprocals}
$$
\dS \triangleq \sum_{p\mid\Phi_e(n,w)}\frac1{s_p}.
$$
The point of this quantity is that when taking logarithms the multiplicative size of $\Phi_e(n,w)$ becomes
an additive constraint on how many exponent classes can occur in the short localized interval.
The next lemma provides lower and upper bounds on the sum of reciprocals. All logarithms are natural.

\begin{lemma}
\label{lem:section5-Lloyd-mass}
If there exists an $e$-perfect code in $\textup{J}(n,w)$, then
\begin{equation}
\label{eq:section5-mass-bounds}
\frac{\log\Phi_e(n,w)}{\log M}\le \dS<\frac{\log\Phi_e(n,w)}{\log \nminusw}.
\end{equation}
\end{lemma}
\begin{proof}
By Corollary~\ref{cor:prime factors from chain divisitilities} we have that $\Phi_e(n,w)$ is squarefree and hence
$$
\log\Phi_e(n,w)= \hspace{-0.3cm} \sum_{p\mid\Phi_e(n,w)} \hspace{-0.3cm} \log p.
$$
For each summand, \( \nminusw <p^{s_p}\le M\) is equivalent to
$$
\frac{\log \nminusw}{s_p}<\log p\le\frac{\log M}{s_p}.
$$
Summing over all prime divisors of $\Phi_e(n,w)$ yields \cref{eq:section5-mass-bounds}.
\end{proof}

\begin{lemma}
\label{lem:number_mult1}
If there exists an $e$-perfect code in $\textup{J}(n,w)$, then
for at most $e$ primes that divide $\Phi_e(n,w)$, we have that $s_p$ is equal to $1$.
\end{lemma}
\begin{proof}
If $s_p=1$ then $\nminusw<p\le \nminusw +d=M$.
Since $0\le h_i\le d-1 <w$, where $h_i$ was defined in Lloyd's theorem, it follows that
$$
0<w-h_i\le w<p.
$$
Hence, $p$ cannot divide any \(w-h_i\). Since $p$ divides $\Phi_e(n,w)$, the product factorization of $\Phi_e(n,w)$
in \cref{eq:section5-Lloyd-factorization} implies that $p$ divides $\nminusw+h_i+1$ for some $i$. On the other hand
since $M<2 \nminusw$ we have that
$$
\nminusw <\nminusw +h_i+1\le \nminusw+d=M<2 \nminusw <2p.
$$
Therefore, $p=\nminusw+h_i+1$, $p$ has localized exponent $1$, and since all the $h_i$s are distinct, it follows that at most $e$ primes can have localized exponent~$1$.
\end{proof}

We now turn our attention to two results concerning the distribution of primes.
The first one is an improvement to what is known as Bertrand's postulate and Bertrand-Chebyshev theorem that
for every integer $x>1$ there is at least one prime between $x$ and $2x$. There are a few improvements that can be
used for our purpose and we will use the following one.

\begin{proposition}[Theorem~4 in \cite{Axl18}]
\label{prop:quoted-axler-short-interval}
For every $x>1$, there is a prime number $p$ such that
$$
x<p\le x\left(1+\frac{198.2}{\log^4 x}\right).
$$
\end{proposition}

In the following lemma, the comparison primes follow from \cref{lem:block-congruence-mod-p},
\cref{lem:modulo-p-criterion}, and \cref{cor:moving-prime-detector}.
The finite table for $e\le396$ is used to establish the bound $w > e^4$ in \cref{lem:section5-weight-amplification}.

\begin{lemma}
\label{lem:section5-comparison-primes}
For every $e\ge10$, where $e\not\in\{12,13,20,21\}$, if an $e$-perfect code exists,
then there are two odd primes $p_e<q_e$ in~$\cP_e$ and
\begin{equation}
\label{eq:section5-cube-gap-table}
q_e^3-p_e^3>e\,p_e^{3/2}.
\end{equation}
In particular, when $e\le 396$, they may be chosen from the table below.

$$
\begin{array}{c|r|r|r||c|r|r|r}
 e\text{-range}&p_e&q_e&(q_e^3-p_e^3)/p_e^{3/2}&e\text{-range}&p_e&q_e&(q_e^3-p_e^3)/p_e^{3/2}\\ \hline
 10&11&13&23.73\ldots&53\le e\le66&67&71&104.20\ldots\\
 14\le e\le16&17&19&27.76\ldots&67\le e\le82&83&89&176.12\ldots\\
 17\le e\le18&19&23&64.09\ldots&83\le e\le100&101&107&191.85\ldots\\
 22&23&29&110.80\ldots&101\le e\le117&127&131&139.53\ldots\\
 23\le e\le28&29&31&34.59\ldots&118\le e\le147&149&157&308.96\ldots\\
 29\le e\le30&31&37&120.86\ldots&148\le e\le190&191&197&256.66\ldots\\
 32\le e\le40&41&43&40.32\ldots&191\le e\le240&241&251&485.31\ldots\\
 41\le e\le43&47&53&139.82\ldots&241\le e\le310&311&317&323.59\ldots\\
 44\le e\le52&53&59&146.43\ldots&311\le e\le396&397&409&739.19\ldots
\end{array}
$$
\end{lemma}
\begin{proof}
For $10\le e\le 396$, where $e\not\in\{12,13,20,21\}$, since an $e$-perfect code exists,  \cref{cor:square-factor-e-plus-one} already
excludes $e=11,19,31$. For each entry $e$ of the table, $p_e$ and $q_e$ exceed $e$ and are at most $(4e+2)/3$.
Hence $p_e$ and $q_e$ are contained in~$\cP_e$ for all entries.
Further, $(q_e^3-p_e^3)/p_e^{3/2} $ exceeds $e$, implying that \cref{eq:section5-cube-gap-table} holds.

For $e\ge 397$, we use \cref{prop:quoted-axler-short-interval} and define
$$
r(x) \triangleq 1+\frac{198.2}{\log^4x}.
$$
Since the function $r(x)$ decreases, $r(e)\le r(397)<1.16$. Applying \cref{prop:quoted-axler-short-interval} on $x=e$ and on $x=1.17e$ yields
two primes $p_e$ and $q_e$ such that
$$
e<p_e\le e \cdot r(e)<1.17e<q_e\le 1.17e \cdot r(1.17e)\le 1.17e \cdot r(1.17\cdot 397)<4e/3.
$$
Thus $\cP_e$ contains two primes $p_e,q_e$ for every $e\ge 397$. Moreover,
$$
q_e^3-p_e^3>e^3(1.17^3-r(e)^3),\qquad e\,p_e^{3/2}\le e^{5/2}r(e)^{3/2}.
$$
The inequality $q_e^3-p_e^3>e\,p_e^{3/2}$ holds when $\sqrt{e}(1.17^3-r(e)^3)>r(e)^{3/2}$. When $e=397$ we have
$$
\sqrt{397}(1.17^3-r(397)^3)>1.245>1.241>r(397)^{3/2}.
$$
The left side increases when $e$ increases, and the right side decreases when $e$ increases,
and therefore \cref{eq:section5-cube-gap-table} holds for all ${e\ge397}$.
\end{proof}

\begin{lemma}
\label{lem:compPrimesSmall}
If $e \in \{4,5,6,9,12,13,20,21\}$ and there exists an $e$-perfect code in $\textup{J}(n,w)$, then the following primes divide $\Phi_e(n,w)$:
\newline
$5$ for $e=4$;~~$2,7$ for $e=5$ and $e=6$;~~$5,11$ for $e=9$;~~$2,7,13$ for $e=12$;~~$7,17$ for $e=13$;~~$11,23$ for $e=20$ and for $e=21$.
\end{lemma}
\begin{proof}
All these primes that divide $\Phi_e(n,w)$ are easily verified by \cref{lem:modulo-p-criterion}.

The parameters required for the proof for each $e$ and each prime that divides $\Phi_e(n,w)$ as given in the lemma
are the triples $(p,P,\gamma)$ (as defined in \cref{lem:modulo-p-criterion}) presented in the following table:
$$
\begin{array}{c|c}
 e&(p,P,\gamma)\\
\hline
 4&(5,5,1)\\
 5&(2,2,1),(7,7,2)\\
 6&(2,8,2),(7,7,1)\\
 9&(5,5,1),(11,11,2)\\
 12&(2,16,4),(7,7,2),(13,13,1)\\
 13&(7,7,1),(17,17,4)\\
 20&(11,11,2),(23,23,3)\\
 21&(11,11,1),(23,23,2)
\end{array}.
$$
\end{proof}

The next lemma improves on \cref{prop:lower bound of w}.
Combined with the upper bound on the strength redundancy in \cref{cor:str_redund}, i.e., $d\le e\sqrt w$, this makes the localized interval
$(\nminusw,\nminusw+d]$ sufficiently short
for the sum of reciprocals and exponent-spacing arguments used in the remainder of Section~\ref{sec:complete-exclusion}.

\begin{lemma}
\label{lem:section5-weight-amplification}
If $e \geq 4$ and an $e$-perfect code exists in \textup{J}$(n,w)$, then $w > e^4$.
Moreover, $\bar\omega>15125$ for $e=4$, $\bar\omega>16171$ for $e=5$ and $\bar\omega>2124$ for $e=6$.
\end{lemma}
\begin{proof}
Suppose for a contradiction that $w\le e^4$. By \cref{eq:section5-w-over-z} and since by \cref{cor:str_redund} $d\le e\sqrt w$, we have that
\begin{equation}
\label{eq:section5-quartic-basic-upper}
\nminusw<e^4+e^3,\qquad M=\nminusw +d<e^4+2e^3.
\end{equation}
When $e\ge10$ and $e\not\in\{12,13,20,21\}$, by \cref{lem:section5-comparison-primes} there exist two primes $p_e,q_e\in\cP_e$
with $p_e<q_e \leq (4e+2)/3$. This implies using also \cref{prop:lower bound of w} that
%$$
\begin{equation}
\label{eq:qe2-upper}
p_e^2<q_e^2 < 2e^2 < \nminusw,
\end{equation}
%$$
and hence the localized exponents $s_{p_e},s_{q_e} \geq 3$. On the other hand, \cref{eq:section5-quartic-basic-upper} gives
%$$
\begin{equation}
\label{eq:qe4-lower}
M<e^4+2e^3<(e+1)^4\le p_e^4<q_e^4,
\end{equation}
%$$
and therefore $s_{p_e},s_{q_e}=3$. The cubes of these primes lie in $(\nminusw,\nminusw+d]$, and hence we have
$$
q_e^3-p_e^3<d\le e\sqrt{\nminusw}<e\,p_e^{3/2}.
$$
This contradicts \cref{eq:section5-cube-gap-table} and thus $w > e^4$.

When $4\le e<10$, \cref{cor:square-factor-e-plus-one} excludes $e=7,8$. Hence, it remains to consider eight cases $e=4,5,6,9,12,13,20,21$.
These cases are solved in Appendix A.
\end{proof}

In the sequel we need the following simplified version of a result on the lower bound of linear forms of logarithms,
taken from Laurent~\cite{Lau08}.
\begin{proposition}[Corollary~2 and Table~1 in \cite{Lau08}, specialized to prime case]
\label{prop:quoted-laurent-cor2}
Let $p$ and $q$ be two distinct primes, let $u$ and $v$ be two positive integers, and
let $\Lambda=u\log p-v\log q$. Define $\ell(p)=\max\{1,\log p\}$ for primes $p,q$.
If
$$
B \triangleq \frac{v}{\ell(p)}+\frac{u}{\ell(q)},\quad H\triangleq \max\{\log B+0.38,m\},
$$
then
$$
\log|\Lambda|\ge-C_2 H^2\ell(p)\ell(q),
$$
for each pair \((m,C_2)\) in the following table:
$$
\begin{array}{c|ccccccccccc}
m&10&12&14&16&18&20&22&24&26&28&30\\ \hline
C_2&25.2&23.4&22.1&21.1&20.3&19.7&19.2&18.8&18.4&18.1&17.9.
\end{array}
$$
\end{proposition}

Applying \cref{prop:quoted-laurent-cor2} with $(m, C_2)=(10, 25.2)$, we have
\begin{equation}
\label{eq:F_Lau}
\log |\Lambda|\geq -25.2 H^2 \ell(p) \ell(q)
\end{equation}
with $H= \max \{\log B+0.38,10\}$. Note that $\ell(p)=1$ only when $p=2$.

\begin{lemma}
\label{lem:6divides}
If there exists a $2$-perfect code in $\textup{J}(n,w)$, then $2$ and $3$ divide $\Phi_2(n,w)$.
\end{lemma}
\begin{proof}
\cref{eq-congruence-e+2} implies that $w$ and $\nminusw$ are both even, and \cref{prop:four_divide} implies that
${4\mid(w-1)(w-2)}$ and $4\mid({\bar\omega}-1)(\nminusw-2)$. Hence, $w\equiv {\bar\omega}\equiv2\pmod4$.
Therefore,
$$
\Phi_2 (n,w)=1+w{\bar\omega}+\binom w2\binom {\bar\omega}2\equiv1+0+1\equiv0\pmod2.
$$
\cref{lem:modulo-p-criterion} with $(p,P,\gamma)=(3,3,1)$ also gives $3\mid\Phi_2(n,w)$. Hence $2,3\mid\Phi_2(n,w)$.
\end{proof}

We next apply this effective two-logarithm estimate to cap the odd localized exponents.
\begin{lemma}
\label{lem:section5-general-cap}
If there exists an $e$-perfect code in $\textup{J}(n,w)$, with $e \geq 2$, then every odd prime divisor $p$ of $\Phi_e(n,w)$,
except $p=5$ when $e=4$, satisfies the following
$$
s_p<
\begin{cases}
6000\log e, & e\ge9,\\
8115, & e\in \{2,5,6\}, \text{~or~} e=4 ~(p \neq 5).
\end{cases}
$$
\end{lemma}
\begin{proof}
\cref{cor:square-factor-e-plus-one} excludes \(e=3,7,8\), so it suffices to consider the radii appearing in the statement.
Let $p$ be an odd prime dividing $\Phi_e (n,w)$.
We consider now another prime $q$ different from $p$, such that $q$ divides $\Phi_e(n,w)$.
For this purpose we use the following results:
\begin{itemize}
\item For $e\ge10$ with $e\not\in\{12,13,20,21\}$, choose one $q\ne p$ from the two odd primes supplied by \cref{lem:section5-comparison-primes}.

\item For $e \in \{9,12,13,20,21\}$, choose the smallest odd $q\ne p$ from \cref{lem:compPrimesSmall}.

\item For $e\in \{5,6\}$, take $q=2$ using \cref{lem:compPrimesSmall}.

\item For $e=2$, take $q=2$ using the fact that $2$ and $3$ divide $\Phi_2 (n,w)$ by \cref{lem:6divides}.

\item For $e=4$ and $p \neq 5$, choose $q=5$ by \cref{lem:compPrimesSmall}.
\end{itemize}

Let $u=s_p$ and $v=s_q$.
By \cref{cor:prime factors from chain divisitilities} we have the localization $\bar\omega<p^u, q^v\le \bar\omega+d$ and by \cref{cor:str_redund}
we have that $d \leq e \sqrt{w}$. Moreover, for any real $x>0$ we have that $\log (1+x) <x$
and since $p^u \neq q^v$ we have
\begin{equation}
\label{eq:section5-Lambda-upper}
0 < \abs{u \log p - v \log q} = \abs{\log \frac{\max\{p^u,q^v\}}{\min\{p^u,q^v\}}} < \log\left(1+\frac{d}{\nminusw}\right) <
\frac{d}{\nminusw}\le\frac e{\sqrt{\nminusw}}
\end{equation}
and taking logarithms gives
$$
\log|u\log p-v\log q|<\log e-(\log\nminusw)/2.
$$
Together with the lower bound on $\log|u\log p-v\log q|$ given in \cref{eq:F_Lau} we now have
\begin{equation}
\label{eq:section5-logz-Laurent}
\log \nminusw<2\log e+50.4H^2 \ell(q) \, \log p,
\end{equation}
where the parameters $B$ and $H$ are taken from \cref{prop:quoted-laurent-cor2}. Moreover, since $\nminusw \ge w>e^4$ by
\cref{prop:lower bound of w} and \cref{lem:section5-weight-amplification},
and $d\le e\sqrt w$
by \cref{cor:str_redund}, we have that $\log(1+d/\nminusw)\le \log(1+e/\sqrt{\nminusw})<\log(1+1/e)$,
and therefore the localization $\nminusw<p^u,q^v\le \nminusw+d$ yields direct bounds for $u$ and $v$:
$$
u\le\frac{\log(\nminusw+d)}{\log p}=\frac{\log\nminusw+\log(1+d/\nminusw)}{\log p}<\frac{\log \nminusw +\eta_e}{\log p},
$$
where $\eta_e =\log(1+1/e)$; similarly $v<\frac{\log \nminusw +\eta_e}{\log q}$.

We next claim that $H=10$. Indeed, \cref{eq:section5-logz-Laurent} with $H\ge 10$ gives
$$
B=\frac{u}{\ell(q)}+\frac{v}{\log p}< \frac{\log {\bar\omega}+\eta_e}{\log p}\left(\frac{1}{\ell(q)}+\frac{1}{\log q}\right)< 50.4H^2\left(1+\frac{\ell(q)}{\log q}\right)+ \frac{4\log e+2\eta_e}{\log p\log q}.
$$
For $e\ge10$ with $e\not\in\{12,13,20,21\}$, we can always find a detected prime $q\neq p$ larger than $e$ by \cref{lem:section5-comparison-primes}. Hence $\log q>\log e, \ell(q)/\log q=1$, and thus
$$
B<100.8H^2+\frac{4\log e+2\eta_e}{\log 3\log e}< 100.8H^2+4< 101H^2,
$$
where the second inequality holds because $(4\log e+2\eta_e)/(\log 3\log e)$ is decreasing for $e\ge10$ and reaches
the maximum value at most $4$ at $e=10$. The last inequality holds because $H\ge10$ and $4\le0.04H^2$.
For $e=2,5,6,9,12,13,20,21$ or $e=4$ with $p\neq5$, the right side
$$
50.4H^2\left(1+\frac{\ell(q)}{\log q}\right)+ \frac{4\log e+2\eta_e}{\log p\log q}
$$
reaches the maximum value at most $124H^2$ at $e=6,p=3,q=2$ by the choices of $q$. In both cases, $B<124H^2$.
However, if $H>10$, then $B=\exp(H-0.38)$, whereas $\exp(x-0.38)/x^2>150$ at $x=10$ and increases for $x>2$. Hence $B>150H^2$,
contradicting the previous $B<124H^2$. Therefore, $H=10$ and by \cref{eq:section5-logz-Laurent} we have

\begin{equation*}
%\label{eq:section5-u-basic}
u<\frac{\log \nminusw +\eta_e}{\log p}< 5040\ell(q)+\frac{2\log e+\eta_e}{\log p}.
\end{equation*}

When $e\ge9$, \cref{lem:section5-comparison-primes} and \cref{lem:compPrimesSmall} induce a comparison prime $q\ne p$ with $q\le (4e+2)/3$, while $p\ge3$. Hence,
\begin{align*}
u<&5040\log q+\frac{2\log e+\eta_e}{\log p}< 5040\log\frac{4e+2}{3}+\frac{2\log e+\eta_e}{\log 3}\\
=&\left(5040+\frac{2}{\log 3}\right)\log e+ 5040\log\frac{4+2/e}{3}+ \frac{\log(1+1/e)}{\log 3}<5042\log e+ 1722.6\\
<&6000\log e,
\end{align*}
where $5040\log\frac{4+2/e}{3}+ \frac{\log(1+1/e)}{\log 3}$ is decreasing for $e\ge9$ and less than $1722.6$ at $e=9$, and $1722.6<784\log9\le 784\log e$.

When $e=2,5,6$, or $e=4$ with $p\neq5$, the bound $5040\ell(q)+(2\log e+\eta_e)/\log p$
reaches the maximum value at $q=5,e=4,p=3$ by the choices of $q$. That is,
$$
u< 5040\log5+\frac{2\log4+\eta_4}{\log3}< 8115.
$$
\end{proof}

The next elementary packing lemma is used for upper-bounding the sum of reciprocals.
\begin{lemma}
\label{lem:section5-exponent-packing}
Let $\cE$ be a set of integers such that $g+1 \leq u \leq N$ for every $u \in \cE$,
and $\gcd(u,v)\le g$ whenever $u$ and $v$ are distinct integers in $\cE$ for some positive integer $g$. Then the following hold.
\begin{itemize}
  \item[\textup{(1)}] When $g=1$, then $$\sum_{u\in\cE}\frac{1}{u}\le\sum_{\substack{2\le p\le N\\p\text{ prime}}}\frac{1}{p}.$$
  \item[\textup{(2)}] When $g=2$, then $$\sum_{u\in\cE}\frac{1}{u}\le \frac{1}{4}+\sum_{\substack{3\le p\le N\\p\text{ prime}}}\frac{1}{p}.$$
  \item[\textup{(3)}] When $g=3$, then $$\sum_{u\in\mathcal E}\frac1u\le
\frac14+\frac16+\frac19+ \sum_{\substack{5\le p\le N\\p\text{ prime}}}\frac1p.$$
\end{itemize}
\end{lemma}
\begin{proof}
If $u\in\mathcal E$ has a prime divisor greater than $g$, then choose one such prime divisor $r_u$.
These chosen primes are distinct because a shared prime would divide the
gcd of two associated elements of $\cE$, contradicting that $\gcd(u,v)\le g$.
The contribution from $u$ satisfies $1/u\le1/r_u$, and altogether the contributions are at most $\sum_{\substack{g+1\le p\le N,p\text{ prime}}}\frac1p$.
When $g=1$, all elements in $\mathcal E$ have been counted and we have (1).

When $g=2$, every remaining element $v\in \cE$ has only $2$ as prime factor. Since $v\geq 3$, $v$ must be divisible by $4$ and $v$ is unique since $\gcd(u,v)\le 2$. Hence the contribution of $v$ is at most $1/4$ and we have (2).

Similarly when $g=3$, every remaining element $v \in \cE$  has only $2$ and $3$ as prime factors, and since $v \geq 4$, it is divisible by $4$, $6$, or $9$.
None of the integers $4$, $6$, and $9$ can divide two distinct elements $u,v \in \cE$ since $\gcd(u,v)\le3$.
Choosing one such divisor for each remaining element bounds their contributions by $1/4+1/6+1/9$ and we have (3).
\end{proof}

The following result of~\cite{RoSh62} is required in the sequel.
\begin{proposition}[Equation (3.18) in Theorem~5 in \cite{RoSh62}]
\label{prop:quoted-RS-reciprocal}
If $B\approx0.261497212847643 < 0.2615$, then for any integer $x\geq 286$,
\begin{equation*}
\sum_{\substack{p\le x\\p\text{ prime}}}\frac1p <\log\log x+B+\frac{1}{2\log^2x}.
\end{equation*}
\end{proposition}

Consequently, \cref{prop:quoted-RS-reciprocal} implies that for $x\geq 286$ we have
\begin{equation}
\label{eq:section5-Dusart-simple}
\sum_{\substack{p\le x\\p\text{ prime}}}\frac1p<\log\log x+0.278, \qquad where ~~ x\ge286,
\end{equation}
since the term $B+\frac{1}{2\log^2x}$ decreases when $x$ increases and is less than $0.278$ at $x=286$.

Before proving the main result, we need the following constraint on big localized exponents.
\begin{lemma}
\label{clm:smallgcd} Let $g_e=3$ when $e\ge 9$, $g_e=2$ when $e=5,6$, and $g_e=1$ when $2\le e\le 4$. If there exists an $e$-perfect code in $\textup{J}(n,w)$, then any two odd prime divisors $p<q$ of $\Phi_e (n,w)$  satisfy $\gcd(s_p,s_q)\le g_e$.  Consequently, for every fixed localized exponent $s\ge g_e+1$,
at most one odd prime divisor  of $\Phi_e (n,w)$ can occur. %{\color{red} (We may state this claim as a lemma before Theorem 6.1)}
\end{lemma}
\begin{proof}
It suffices to show the first claim since $s_p=s_q=s\ge g_e+1$ gives $\gcd(s_p,s_q)>g_e$.
Suppose for a contradiction, that $g=\gcd(s_p,s_q)\ge g_e+1$ for some pair $(p,q)$. Write $p^{s_p}=x^g$, $q^{s_q}=y^g$.
The positive integers $x,y$ are odd and distinct, hence $| x-y|\ge2$.
Since $p^{s_p} ,q^{s_q} > \nminusw$, by the
mean-value theorem,
$$
\theta \triangleq |p^{s_p}-q^{s_q}|=|x^g-y^g|>2g\bar\omega^{1-1/g}.
$$
By \cref{lem:section5-weight-amplification}, we have $\bar\omega>e^4$ for $e\ge 4$. So when $e\geq 9$,
$g\ge 4$ and then $\theta >8\bar\omega^{3/4}>e\sqrt{\bar\omega}$;  when $e=5,6$, we have $g\ge 3$ and
then $\theta >6\bar\omega^{2/3} >6e^{2/3}\sqrt{\bar\omega}>e\sqrt{\bar\omega}$;
when $2 \leq e \leq 4$, we have $g\ge 2$ and then $\theta >4\sqrt{\bar\omega}\ge e\sqrt{\bar\omega}$.
All these cases give $\theta >e\sqrt{\bar\omega}\geq d$, %(when $e=1$, $\theta >4\sqrt{\bar\omega}>d$ by \cref{cor:str_redund}),
contradicting the localization $\theta < d$.
\end{proof}

\begin{theorem}
\label{thm:section5-unified-tail}
There is no nontrivial $e$-perfect code in any Johnson scheme for $e \ge 3$.
\end{theorem}
\begin{proof}
\cref{cor:square-factor-e-plus-one} excludes radii $e=3,7,8$.
Assume, for the contradiction, that there exists an $e$-perfect code in J$(n,w)$ for some $e\ge 4$.
We then prove the nonexistence by deriving incompatible bounds on the sum of reciprocals $\dS$.

First, \cref{thm:section5-Lloyd-factorization} yields a lower bound on $\dS$.
By \cref{lem:section5-weight-amplification,eq:section5-w-over-z,cor:str_redund}, respectively, we have that
\begin{equation*}
\nminusw \ge w>e^4,\qquad \frac w\nminusw>\frac e{e+1}, \qquad \text{and} \qquad d\le e\sqrt{\nminusw}, \qquad \text{respectively}.
\end{equation*}
Therefore, $d/\nminusw\le e/\sqrt{\nminusw}<1/e$, $M=\nminusw+d<\nminusw(1+1/e)$, and $w-h_i>w-d>c_e\nminusw$,
where $c_e=\frac e{e+1}-\frac 1e$ and $h_i$ is defined in \cref{thm:section5-Lloyd-factorization}.
By \cref{thm:section5-Lloyd-factorization} we have that $\Phi_e(n,w)>(c_e\nminusw^2)^e/(e!)^2$, and hence
\begin{equation}
\label{eq:first_part_L}
\log\Phi_e (n,w)>2e\log \nminusw+e\log c_e-2\log(e!).
\end{equation}
We turn now to \cref{eq:section5-mass-bounds} to obtain a lower bound on $\dS/e$ using \cref{eq:first_part_L},
the well-known Stirling formula $\log(e!)\le e\log e-e+\frac12\log e+1$ and $\log M<\log \nminusw+\log(1+1/e)$. These lead to
the following equalities and inequalities:
\begin{equation}
\label{eq:section5-strict-lower}
\begin{aligned}
\frac{\dS}{e} \ge \frac{\log\Phi_e(n,w)}{e\log M}>& \frac{2\log \nminusw+\log c_e-2\log e+2-(\log e+2)/e}{\log \nminusw+\log(1+1/e)}\\
=& 2+\frac{\log c_e-2\log e+2-(\log e+2)/e-2\log(1+1/e)}{\log \nminusw+\log(1+1/e)}.
\end{aligned}
\end{equation}
Since $\log c_e-2\log e+2-(\log e+2)/e-2\log(1+1/e)<-2\log e+2<0$, the right side of \cref{eq:section5-strict-lower}
is increasing when $\nminusw$ is increasing and its minimum value is at $\nminusw=e^4$. Therefore,
\begin{equation}
\label{eq:last_forS}
\begin{aligned}
\frac{\dS}{e}>& \frac{2\log \nminusw+\log c_e-2\log e+2-(\log e+2)/e}{\log \nminusw+\log(1+1/e)}> \frac{6\log e+\log c_e+2-(\log e+2)/e}{4\log e+\log(1+1/e)}\\
=& \frac{3}{2}+ \frac{\log c_e+2-(\log e+2)/e-3\log(1+1/e)/2}{4\log e+\log(1+1/e)}.
\end{aligned}
\end{equation}
When $e \geq 4$ \cref{eq:last_forS} implies that
\begin{equation}
\label{eq:section5-tail-lower}
\dS>\frac{3}{2}e.
\end{equation}
Moreover, substituting $\bar\omega>15125$, $\bar\omega>16171$ and $\bar\omega>2124$, respectively,
from \cref{lem:section5-weight-amplification} into \cref{eq:section5-strict-lower} gives
\begin{equation}
\label{eq:section5-tail-lower-456}
\dS>1.729e,\quad e=4;\qquad \dS>1.72e,\quad e=5;\qquad \dS>1.62e,\quad e=6, \quad \text{respectively.}
\end{equation}

To obtain an incompatible upper bound on $\dS$, recall that $\dS=\sum_{p\mid\Phi_e(n,w)}\frac1{s_p}$ is the
sum of reciprocals of exponents. We separate the contributions of different exponent classes in $\dS$.
Since by the localization $2^{s_2}>\nminusw$ and by \cref{lem:section5-weight-amplification} $\nminusw>e^4$,
it follows that the contribution of $2$, when $2$ divides $\Phi_e(n,w)$, to $\dS$
is $1/s_2<\log2/\log \nminusw<\log2/(4\log e)$.
For odd primes, note that small localized exponents contribute the most, so for different radii $e$, we count the possible
odd primes for small localized exponents separately
and then handle all the large localized exponents together.
By \cref{lem:number_mult1} we have that at most $e$ primes have exponent $1$.
Let $N_i$ be the number of odd prime divisors with localized exponent $i$.
If $N_i \leq 1$, then the contribution of the associated localized exponent to $\dS$ is either $0$ or $1/i$.
If $N_2 > 1$ and their primes are $p_1<\cdots<p_{N_2}$, then for each $1\le j<N_2$, the mean-value theorem implies that
$$
p_{j+1}^2-p_j^2>2p_j(p_{j+1}-p_j)\ge 4p_j >4\sqrt{\nminusw},
$$
where the last inequality is from the localization $\nminusw<p_j^2$ in \cref{cor:prime factors from chain divisitilities}.
Summing over all $N_2-1$ gaps $p_{j+1}^2-p_j^2$ gives
$$
4(N_2-1)\sqrt{\nminusw}<p_{N_2}^2-p_1^2< d\le e\sqrt{\nminusw},
$$
where $p_{N_2}^2-p_1^2< d$ is due to the fact that $p_j^2$ lie in the localized interval $(\nminusw,\nminusw+d]$ whose length is $d$,
while by \cref{cor:str_redund} $d\le e\sqrt{\nminusw}$. Thus, $N_2<1+e/4$
and hence the contribution from odd primes with localized exponent $2$ to $\dS$ is $N_2/2$ which is less than $1/2+e/8$.

If $N_3 >1$ and their bases are $q_1<\cdots<q_{N_3}$,
then similar to the localized exponent $2$ case,
$$
q_{j+1}^3-q_j^3>3q_j^2(q_{j+1}-q_j)\ge 6q_j^2 >6\nminusw^{2/3},
$$
where the last inequality is from the localization $\nminusw<q_j^3$. Summing over all $q_{j+1}^3-q_j^3$ gives
$$
6(N_3-1)\nminusw^{2/3}<q_{N_3}^3-q_1^3< d\le e\sqrt{\nminusw}.
$$
Hence by $\nminusw>e^4$, $N_3<1+e\nminusw^{-1/6}/6<1+e^{1/3}/6$. Therefore, the contribution from odd primes
with localized exponent $3$ to $\dS$ is $N_3/3$ which is less than $1/3+e^{1/3}/18$.

Now we consider different radii $e$. For $e\geq 9$, each localized exponent $s_p<6000 \log e$ for an odd prime $p$
by \cref{lem:section5-general-cap}. \cref{clm:smallgcd} implies that at most one odd prime divisor
of $\Phi_e (n,w)$ can occur for every fixed localized exponent  $s_p\ge4$ and $\gcd(s_p,s_q)\leq 3$.
Applying \cref{lem:section5-exponent-packing} and \cref{eq:section5-Dusart-simple},
the contribution of odd-prime localized exponents at least $4$ to $\dS$ is
\begin{align*}
\sum_{\substack{p\mid\Phi_e(n,w)\\p\text{ odd},\ s_p\ge4}}\frac1{s_p}&\le
\frac14+\frac16+\frac19+ \sum_{\substack{5\le p\le 6000\log e\\p\text{ prime}}}\frac1p<\frac14+\frac16+\frac19+ \bigl(\log\log(6000\log e)+0.278-\frac12-\frac13\bigr)\\
&< \log\log(6000\log e).
\end{align*}
Altogether,
\begin{equation}
\label{eq:section5-tail-upper}
\begin{aligned}
\dS=&\sum_{p\mid\Phi_e(n,w)}\frac1{s_p}\le \sum_{s_p=1}\frac1{s_p}+\sum_{\substack{p\text{ odd}\\ s_p=2}}\frac1{s_p}
+\sum_{\substack{p\text{ odd}\\ s_p=3}}\frac1{s_p}+\sum_{\substack{p\text{ odd}\\ s_p\ge4}}\frac{1}{s_p}+\frac1{s_2}\\
<&e+\left(\frac 12+\frac e8\right)+\left(\frac 13+\frac{e^{1/3}}{18}\right)+\log\log(6000\log e)+\frac{\log2}{4\log e}\\
=&e+\frac e8+\frac{e^{1/3}}{18}+\log\log(6000\log e)+\frac{5}{6}+\frac{\log2}{4\log e}\\
=&\left( 1+\frac18+\frac{1}{18e^{2/3}}+\frac{\log\log(6000\log e)}{e}+\frac{5}{6e}+\frac{\log2}{4e\log e} \right) e< 1.49e,
\end{aligned}
\end{equation}
contradicting \cref{eq:section5-tail-lower}. Here the last inequality holds because $e \geq 9$ and the main part in the brackets is decreasing with $e$.
It remains to consider $e=4,5,6$.

When $e=5,6$, each localized exponent $s_p<8115$ for odd prime $p$ by \cref{lem:section5-general-cap}.
Similarly applying \cref{clm:smallgcd},  \cref{lem:section5-exponent-packing} and \cref{eq:section5-Dusart-simple}, we have
$$
\sum_{\substack{p\text{ odd}\\ s_p\ge3}}\frac1{s_p}\le \frac14+ \sum_{\substack{3\le p< 8115 \\ p\text{ prime}}}\frac1p< \frac14+\bigl(\log\log8115+0.278-\frac12\bigr)= \log\log8115+0.028.
$$
Altogether,
$$
\dS= \sum_{s_p=1}\frac1{s_p}+\sum_{\substack{p\text{ odd}\\ s_p=2}}\frac1{s_p}+\sum_{\substack{p\text{ odd}\\ s_p\ge3}}\frac1{s_p}+\frac1{s_2}
<e+\left(\frac 12+\frac e8\right)+\bigl(\log\log8115+0.028\bigr)+\frac{\log2}{4\log e}.
$$
Hence, $\mathbb S<1.70e$ when $e=5$, while $\mathbb S<1.60e$ when $e=6$, contradicting \cref{eq:section5-tail-lower-456}.

When $e=4$, \cref{lem:section5-general-cap} gives $s_p< 8115$ when $p\neq 2,5$.
Similarly applying \cref{clm:smallgcd},  \cref{lem:section5-exponent-packing} and \cref{eq:section5-Dusart-simple}, we have
$$
\sum_{\substack{p\text{ odd} \\ s_p \geq 2,p\neq5}}\frac1{s_p}\le \sum_{\substack{p< 8115\\ p\text{ prime}}}\frac{1}{p}< \log\log8115+0.278.
$$
Together with $1/s_5<\log5/\log \nminusw<\log5/(4\log e)$ and the possible $1/s_2<\log2/(4\log e)$,
$$
\dS\le \sum_{s_p=1}\frac1{s_p}+\sum_{\substack{p\text{ odd}\\ s_p \geq 2,p\neq5}}   \frac1{s_p}+\frac{1}{s_2} + \frac{1}{s_5}
<e+\bigl(\log\log8115+0.278\bigr)+\frac{\log10}{4\log e} < 1.723e.
$$
Hence, $\mathbb S<1.723e$, contradicting \cref{eq:section5-tail-lower-456}.
\end{proof}

\section{There Are No \texorpdfstring{$1$}--Perfect Codes}
\label{sec:no_1_perfect}

This section is devoted to $1$-perfect codes in J$(n,w)$.
All the claims, except for the final conclusion, are under the assumption that there exists a $1$-perfect code in J$(n,w)$.
Radius $1$ was extensively studied, first by Martin~\cite{Mar92} who considered the parameters as given in \cref{prop:Martin}. After that
in~\cite{EtSc04} the authors used polynomial and regularity techniques in \cref{prop:main_t_regular} and demonstrated them
first for radius $1$. The same results were obtained in~\cite{Etz06} using moments and were further obtained and analyzed in~\cite{SiEt10},
from which we will use now some results. Finally, \cref{prop:no1large} implies that we have to consider only $n \geq 2^{250}$.
 We start by stating the following proposition.

\begin{proposition}[Theorems~2 and~3, including their proofs, in \cite{SiEt10}]
\label{prop:s7-SE-original}
If there exists a $1$-perfect code in $\textup{J}(2w+\delta,w)$, then
$\delta<w/11$ and the strength redundancy $d\ge12$ .

Moreover, if $K=w-d+1$ then $\Phi_1(n,w)=MK$, where $M=\bar{\omega}+d>K$.
Finally, $\delta$ can be written in terms of the weight $w$ and the strength redundancy $d$ as follows:
$$
\delta = \frac{w-d^2 +d -1}{d-1}.
$$
\end{proposition}

Since $\Phi_1(n,w)$ is squarefree, it follows from \cref{prop:s7-SE-original} that $\gcd (M,K)=1$.

\begin{corollary}
\label{cor:present_w}
If there exists a $1$-perfect code in $\textup{J}(n,w)$, then
\begin{equation}
\label{eq:eq_w_nw}
w=(d-1)(\delta+d)+1, \qquad \nminusw = d(\delta +d-1)+1 > d(d-1);
\end{equation}
\begin{equation}
\label{cor:eq_dK}
dK = (d-1)\nminusw +1, \qquad d^2K-(d-1)^2M=d(d-1)\delta+2d-1>0.
\end{equation}
\end{corollary}

\begin{proposition}[Section~6 in \cite{Etz96}]
\label{prop:s7-modulo12}
If there exists a $1$-perfect code in $\textup{J}(n,w)$ then $w,\ \bar\omega\equiv1\pmod6$.
\end{proposition}

\cref{prop:s7-modulo12} implies that $\Phi_1 (n,w)=1+w\bar\omega\equiv2\pmod6$, and therefore $2$ divides $MK$. This property will be used
in this section.

\begin{lemma}
\label{lem:s7-coprime}
For distinct prime divisors $p,q$ of $MK$, one has $\gcd(s_p,s_q)=1$
and localized exponent $1$ occurs if and only if $M$ is prime; in
that case, $M$ is the only prime divisor with localized exponent $1$.
\end{lemma}
\begin{proof}
By \cref{cor:present_w} we have that $\nminusw > d(d-1)$ and for $d > 3$ it implies that $2\sqrt{\bar\omega}\ge d$.
Note that from the proof of \cref{clm:smallgcd}, when $e=1$, any two prime divisors (not necessarily odd) satisfy $\gcd(s_p,s_q)= 1$ since $|x-y|\ge1$ and $\theta =|x^g-y^g| >g\bar\omega^{1-1/g}\ge 2\sqrt{\bar\omega}\ge d$.
Also, from the proof of \cref{lem:number_mult1}, when $e=1$,
the prime divisor with localized exponent $1$ occurs if and only if $\nminusw +h_1+1=\nminusw+d=M$ is a prime.
\end{proof}

We continue with the sum of reciprocals defined in \cref{sec:complete-exclusion}, and define the product of the localized exponents:
$$
\dS=\sum_{p\mid \Phi_1(n,w)}\frac1{s_p},\qquad \cQ=\prod_{p\mid \Phi_1(n,w)}s_p.
$$
It is easy to see that $\dS\cQ$ is a positive integer.

The following lemma shows that $\dS$ is very close to $2$ and gives a lower bound of $\dS$.
\begin{lemma}
\label{lem:s7-two-masses}
If there exists a $1$-perfect code in $\textup{J}(n,w)$ then
\begin{equation}
\label{eq:s7-B-error}
0<|\dS-2|<\frac{2}{(d-1)\log M}, \qquad \dS>1.9989.
\end{equation}

\end{lemma}
\begin{proof}
\cref{lem:section5-Lloyd-mass} gives
$$
\frac{\log MK}{\log M}\le \dS <\frac{\log MK}{\log {\bar\omega}}.
$$
Using $\log(1+x)<x$ for $x>0$ and \eqref{cor:eq_dK}, we have that
$$ \log\frac{M}{K}< \log\left(\frac{d}{d-1}\right)^2=2\log\left(1+\frac{1}{d-1}\right)< \frac{2}{d-1},
$$
and thus
$$
\mathbb S\ge \frac{\log MK}{\log M}= 2-\frac{\log(M/K)}{\log M}> 2-\frac{2}{(d-1)\log M}.
$$
On the upper side, $MK=1+w\bar\omega\le1+\bar\omega^2$, $\log M<2\log\bar\omega$ and $d-1<\bar\omega^2$ give
$$
\mathbb S<\frac{\log MK}{\log {\bar\omega}}\le \frac{\log(1+\bar\omega^2)}{\log {\bar\omega}}= 2+\frac{\log(1+1/\bar\omega^2)}{\log {\bar\omega}}<2+\frac{1}{\bar\omega^2\log {\bar\omega}}< 2+\frac{2}{(d-1)\log M}.
$$
Therefore, $|\dS-2|<\frac{2}{(d-1)\log M}$. We also claim that $\dS-2\neq0$. Otherwise, we first note that $s_2\ge 250$ since $2^{s_2}>\bar\omega\geq 2^{249}$ by the localization and \cref{prop:no1large}. Then multiplying $\dS=2$ by $\mathcal Q$ gives $\dS\mathcal Q=2\mathcal Q\equiv0 \pmod{s_2}$, while by the definitions of $\dS$ and $\mathcal Q$, $\dS\mathcal Q\equiv\mathcal Q/s_2\not\equiv0 \pmod{s_2}$ by the coprimality in \cref{lem:s7-coprime}, a contradiction.

Finally, $d\ge12$ by \cref{prop:s7-SE-original} and $\log M>\log\bar\omega\ge249\log2$ by \cref{prop:no1large} imply
$$ \mathbb S> 2-\frac{2}{(d-1)\log M}> 1.9989.
$$
\end{proof}

Rosser and Schoenfeld~\cite{RoSh62} defined $\psi(x)$ as the logarithm of the least common
multiple of all positive integers not exceeding $x$ (and defined $\psi(x)=0$ for $x<2$).
For a positive integer $k$, this definition gives
$\psi(k)=\log\operatorname{lcm}(1,\ldots,k)$.

\begin{proposition}[Theorem~12 in \cite{RoSh62}]
\label{prop:s7-RS-psi}
The quotient $\psi(x)/x$ takes its maximum at $x=113$, and
\begin{equation*}
%\tag{3.35}
%\label{eq:RS-335}
\psi(x)<1.03883x\qquad \text{for} \qquad x>0.
\end{equation*}
\end{proposition}

For an upper bound on $\dS$, we estimate the largest localized exponent of odd prime divisors by repeatedly applying
an upper bound on $\log(d-1)$ and any upper bound, $\lexpo$, on these exponents, in the following two lemmas in turn.

\begin{lemma}
\label{lem:s7-constraints}
If $p$ divides $\Phi_1(n,w)$ and $s_p\ge2$, then $p < (d-1)^2$. If $s_p\le \lexpo$ for every odd prime $p\mid\Phi_1(n,w)$, then
\begin{equation}
\label{eq:s7-D-cap}
\log(d-1)<\psi(\lexpo)+\log\frac{2}{\log2}  <1.03883 \lexpo+1.06.
\end{equation}
\end{lemma}
\begin{proof}
If $p$ divides $\Phi_1(n,w)$ and $s_p\ge2$, then $p^{s_p}={\bar\omega}+\rho_p$, where $1 \le \rho_p\le d$.
The endpoint $\rho_p=d$ would imply $p^{s_p}={\bar\omega}+d=M\mid \Phi_1(n,w)$, contradicting the squarefreeness. Therefore, $1\le \rho_p\le d-1$ and $p^{s_p}<M$. If $p\mid M$, then $p\mid M-p^{s_p}=d-\rho_p$, where $1\le d-\rho_p\le d-1$. Hence, $p\le d-1<(d-1)^2$. If $p\nmid M$, then $p\mid K$.
Substituting $\nminusw=p^{s_p}-\rho_p$ into \cref{cor:eq_dK} gives
\begin{equation*}
dK=(d-1)p^{s_p}-\bigl((d-1)\rho_p-1\bigr).
\end{equation*}
Modulo $p$ we get $p\mid (d-1)\rho_p-1$, and thus $p\le(d-1)\rho_p-1<(d-1)^2$.

By \cref{lem:s7-two-masses} and that $\mathcal Q|\dS-2|$ is an integer, we have
$$
1\le\mathcal Q|\dS-2| <\frac{2\mathcal Q}{(d-1)\log M}.
$$
The localized interval implies that $s_2\log2\le\log M$, which implies that
$$
d-1<\frac{2 \mathcal Q}{\log M} \le\frac{2\mathcal Q}{s_2\log2}.
$$
Taking logarithms we get that
\begin{equation}
\label{eq:s7-denominator}
\log(d-1)<\log \frac{\mathcal Q}{s_2}+\log\frac{2}{\log2}.
\end{equation}
Finally, if all odd-prime localized exponents are at most $\lexpo$, their pairwise coprimality implies that $\mathcal Q/s_2$ divides $\operatorname{lcm}(1,\ldots,\lexpo)$.
Therefore \cref{eq:s7-denominator} and \cref{prop:s7-RS-psi} imply \cref{eq:s7-D-cap}.
\end{proof}

\begin{lemma}
\label{lem:s7-exponent-cap}
Every odd prime $p\mid\Phi_1(n,w)$ satisfies $s_p\le5040$. Moreover,
\begin{equation}
\label{eq:s7-moving}
s_p<\frac{2520\log \nminusw}{\log \nminusw-\log(d-1)-0.001}.
\end{equation}
\end{lemma}
\begin{proof}
Let $p$ be an odd prime divisor. The case $s_p=1$ is immediate, so we assume $s_p\ge2$.
By the localization and the squarefreeness of $M$, we have that $\nminusw < 2^{s_2},p^{s_p} < M$.
Let $\Lambda=s_2 \log2-s_p \log p$. Recall that $(d-1)^2< \nminusw$ and hence
\begin{equation*}
%\label{eq:s7-Lambda-upper}
0<|\Lambda|= \log \frac{\max\{2^{s_2},p^{s_p}\}}{\min\{2^{s_2},p^{s_p}\}} \le \log\frac{M-1}{{\bar\omega}}=
\log\bigl(1+\frac{d-1}{{\bar\omega}}\bigr)< \frac{d-1}{{\bar\omega}}< \frac{1}{\sqrt{\bar\omega}}.
\end{equation*}
Moreover, $(d-1)^2< \nminusw$ and $d\ge12$ imply that $d^2<2{\bar\omega}$, and therefore \cref{prop:no1large} yields that $\log(1+d/{\bar\omega})<\log(1+2/\sqrt{\bar\omega})<0.001$. Applying the proof of \cref{lem:section5-general-cap}
with $\Lambda=s_2 \log2-s_p \log p$ and $\log(1+d/{\bar\omega})<0.001$ yields
$$
s_p< 5040+\frac{0.001}{\log 2}< 5041.
$$
As $s_p$ is an integer, we have that $s_p\le5040$, as required.

Applying the bound $|\Lambda|<(d-1)/{\bar\omega}$ to the proof of \cref{lem:section5-general-cap},
and then $\log {\bar\omega}<2520\log p+ \log(d-1)$, gives
$$
s_p<\frac{\log {\bar\omega}+0.001}{\log p}< 2520+ \frac{\log(d-1)+0.001}{\log p}< 2520+ \frac{\log(d-1)+0.001}{\log \nminusw}s_p,
$$
where the last inequality comes from the localization $p^{s_p}>{\bar\omega}$.
Hence, $s_p<\frac{2520\log \nminusw}{\log \nminusw-\log(d-1)-0.001}$, where the denominator $\log \nminusw-\log(d-1)-0.001$ is positive since we have
$(d-1)^2<\nminusw$ and $\log \nminusw \geq 249 \log 2$.
This proves \cref{eq:s7-moving}.
\end{proof}

Now we apply \cref{lem:s7-constraints,lem:s7-exponent-cap} to sharpen the exponent bounds when $\nminusw$ is large.

\begin{lemma}
\label{lm:rangV}
If  $\nminusw\geq \exp(125000)$, then every odd prime $p\mid\Phi_1(n,w)$ satisfies $s_p=1$ or $24\le s_p\le 2600$.
\end{lemma}
\begin{proof}
If $x \triangleq 125000$, then $x \le\log {\bar\omega}$.
If there exists $d_0$ such that $\log(d-1)<d_0<x-0.001$, then \cref{eq:s7-moving} gives
\begin{equation}
\label{eq:addeq}
s_p<\frac{2520\log \nminusw}{\log \nminusw -d_0-0.001} \le\frac{2520x}{x-d_0-0.001}.
\end{equation}
The second inequality holds since $2520y/(y-d_0-0.001)$ is decreasing for $y>d_0+0.001$.

Now we start from the initial bound $s_p\le 5040$ for every odd prime $p\mid\Phi_1(n,w)$ in \cref{lem:s7-exponent-cap}.
Then  \cref{lem:s7-constraints}
implies the following upper bound on $\log(d-1)$:
$$
\log(d-1) <1.03883\cdot 5040+1.06< 5237.
$$
Taking $d_0=5237<x-0.001$ in \cref{eq:addeq} yields
$$
s_p< \frac{2520\cdot125000}{125000-5237-0.001}< 2631
$$
for every odd prime $p\mid\Phi_1(n,w)$.
We iteratively obtain the second upper bound of $\log(d-1)$:
$$
\log(d-1) <1.03883\cdot 2630+1.06< 2734.
$$
This yields that
$$
s_p< \frac{2520\cdot125000}{125000-2734-0.001}< 2577\leq 2600,
$$
and hence $\log(d-1) <1.03883\cdot 2576+1.06< 2678$.

For the lower bound, suppose $s_p\ge2$. The localized interval implies that $s_p \log p>\log \nminusw$.
%For $p$ that divides $K$ with localized exponent $v=s_p$,
\cref{lem:s7-constraints}
implies that $\log p<2\log(d-1)$. Therefore,
$$
s_p>\frac{\log {\bar\omega}}{2\log(d-1)} >\frac{125000}{2\cdot2678} > 23.
$$
\end{proof}

\begin{lemma}
\label{lem:s7-large}
If a $1$-perfect code exists in $\textup{J}(n,w)$, then $\nminusw<\exp(125000)$.
\end{lemma}
\begin{proof}
Suppose that $\log {\bar\omega}\ge x \triangleq 125000$. \cref{lem:s7-two-masses} gives a lower bound $\dS>1.9989$. So it suffices to show that $\dS<1.9989$.
By \cref{lm:rangV} and \cref{lem:s7-coprime}, all odd-prime localized exponents greater than $1$ lie in $[24,2600]\cap\mathbb Z$ and they are coprime.
We prove that $\dS<1.9989$ as follows.

For each prime $p$, define
$$
b_p=\frac1{24}\text{ when }  p\in \{2,3\}, \quad b_p=\frac{1}{p^2} \text{ when }  p\in \{5,7,11,13,17,19,23\},\quad
b_p=\frac{1}{p} \text{~when~}  p\ge29.
$$
For every odd-prime localized exponent $v>1$, we claim that
$$
\frac{1}{v}\le\sum_{p\mid v}b_p.
$$
If $2$ or $3$ divides $v$, then the claim follows since $v \geq 24$ and hence $1/v\le1/24 =b_2=b_3$.
If a prime $p \geq 29$ divides~$v$, then $v \geq p$ and hence $1/v\le1/p=b_p$,
so again a single weight suffices. Otherwise, every prime divisor belongs to $\{5,7,11,13,17,19,23\}$.
If there is only one such prime~$p$, then since $v \geq 24 >p$, it follows that $v$ is a power of $p$ and hence $v \geq p^2$
which implies that $1/v \leq 1/{p^2}=b_p$.
If there are two distinct prime divisors $p,q$, then since $pq$ divides $v$ we have
$$
b_p+b_q=\frac{1}{p^2}+\frac{1}{q^2}  \ge\frac{2}{pq}\ge\frac{1}{v},
$$
and the claim is proved.

Now let $\Psi=\{2,3,5,7,11,13,17,19,23\}$. Each prime occurs at most once when
we sum over the localized exponents by their coprimality, and no prime divisor exceeds $2600$. As a consequence we have that
\begin{equation}
\label{eq:RCfore1}
\sum_{p\text{ odd},\ s_p>1}\frac{1}{s_p}  \le\sum_{\substack{29\le p\le2600\\p\text{ prime}}}\frac{1}{p}
     +\frac1{12}+\sum_{p\in  \Psi \setminus\{2,3\}}\frac{1}{p^2}.
\end{equation}
Apply the upper bound in \cref{prop:quoted-RS-reciprocal} at $x=2600$. As a result, the right side of \cref{eq:RCfore1} is smaller than
$$
\log\log2600+0.2615+\frac1{2\log^2 2600} -\sum_{p\in\Psi}\frac{1}{p}+\frac1{12}+\sum_{p\in \Psi\setminus\{2,3\}}\frac1{p^2} <0.99888.
$$
By \cref{lem:number_mult1}, at most one exponent equals $1$, while $1/s_2<\log2/\log\bar\omega$. Therefore,
$$
\dS=\sum_{s_p=1}\frac{1}{s_p}+\frac{1}{s_2}+\sum_{p\text{ odd},\ s_p>1}\frac{1}{s_p}< 1+\frac{\log2}{125000}+0.99888<1.9989,
$$
contradicting $\dS>1.9989$, and thus the proof is completed.
\end{proof}

It remains to exclude the bounded range, when $\nminusw <\exp(125000)$ and $n\ge 2^{250}$. The next lemma rules out small odd prime divisors
of the localized exponents. Exponent $2$ is allowed throughout the calculation; its possible contribution
to the sum of reciprocals will be handled separately by a bound of $1/2$.

\begin{lemma}
\label{lem:s7-certificate}
For every integer $250\le u\le180338$ and every odd prime $p\le173$, there is no odd integer $m \ge 3$ such that
\begin{equation}
\label{eq:s7-finite-gap}
0<|2^u-m^p|^2<\min\{2^u,m^p\}.
\end{equation}
\end{lemma}
\begin{proof}
We describe a program that verifies the assertion in \eqref{eq:s7-finite-gap}.
For each pair $(u,p)$, we claim that at most two candidate values of $m$ need to be considered.
Indeed, suppose that \eqref{eq:s7-finite-gap} holds for some odd integer $m\ge3$, and put $t=2^{u/p}$. Then $|2^u-m^p|<2^{u/2}$ and $m^p>2^u-2^{u/2}>2^{u-1}$. Since $p\ge3$ and $u\ge250$, it follows that $t^{p-1},m^{p-1}>2^{u/2}$.
By the mean-value theorem,
\[
2^{u/2}>|2^u-m^p|=|t^p-m^p|\ge p|t-m|\min\{t^{p-1},m^{p-1}\}>p|t-m|2^{u/2},
\]
so $|t-m|<1$. Hence, $m$ must be one of the two consecutive odd integers bracketing $t$. The verification proceeds as follows:
\begin{enumerate}
\item For each pair $(u,p)$, determine $t=2^{u/p}$.

\item Identify the two odd-integer candidates adjacent to $t$, that is, the largest odd integer not exceeding $t$ and the smallest odd
integer not less than $t$.

\item Test whether one candidate satisfies \eqref{eq:s7-finite-gap}. If so, stop; otherwise, proceed to the next pair $(u,p)$.
\end{enumerate}
There are $39$ odd primes $p\le173$ and $39(180338-250+1)=7\,023\,471$ pairs $(p,u)$, giving at most $14\,046\,942$ candidate values of $m$ to check. The actual program includes additional steps to avoid floating-point approximation.
The verification completes successfully and prints \texttt{PASS}, establishing the lemma.
The algorithm with its program is given in Appendix \ref{sec:added}.
\end{proof}

\begin{lemma}
\label{prop:s7-bounded}
If there exists a $1$-perfect code in $\textup{J}(n,w)$, where $n\ge 2^{250}$, then $s_2 > 180338$.
\end{lemma}
\begin{proof}
Assume for the contradiction that there exists a $1$-perfect code in J$(n,w)$ and $s_2 \le180338$.
Further,  $s_2 \ge 250$ since $2^{s_2}>\nminusw \ge 2^{249}$ which is obtained from \cref{prop:no1large}.
We provide an upper bound on $\dS$ by splitting its exponents according to their parity.
At most one exponent is even, since two even exponents would not be coprime.
If there exists such an exponent $v$, then its contribution to $\dS$, $1/v$ is at most $1/2$.

For every odd localized exponent $s_p\ge 2$ with $p$ odd, let $q$ be the smallest prime divisor of $s_p$ and put $m=p^{s_p/q}$. Clearly, $m \geq 3$, and $m^{q}=p^{s_p}$. The two distinct powers $2^{s_2},p^{s_p}$ are contained in $(\nminusw,M)$, and hence
$$
0<|2^{s_2} - m^{q}|^2   \le(d-1)^2<\nminusw<\min\{2^{s_2},m^{q}\}.
$$
Hence, \cref{lem:s7-certificate} rules out $q\le173$. On the other hand, \cref{lem:s7-exponent-cap} yields $q\le s_p\le5040$.
The chosen prime divisors $q$ are distinct because the localized exponents are pairwise coprime.
Moreover, $1/s_p\le1/q$.
Together with the possible exponent $1$ and the contribution $1/s_2$, summing these bounds yields
$$
\dS\le1+\frac{1}{s_2}+\frac{1}{2}+\sum_{\substack{173<p\le5040\\p\text{ prime}}}\frac{1}{p} <1+\frac{1}{250}+\frac{1}{2}+0.494=1.998.
$$
Here, \cref{prop:quoted-RS-reciprocal} at $x=5040$ gives
$$
\sum_{\substack{173<p\le5040\\p\text{ prime}}}\frac{1}{p}< \log\log5040+0.2615+\frac{1}{2\log^2 5040}-\sum_{\substack{p\le173\\p\text{ prime}}}\frac{1}{p}< 2.411402-1.917404< 0.494.
$$
The upper bound of $\dS$ contradicts the lower bound $\dS>1.9989$ in \cref{lem:s7-two-masses}, and therefore $s_2 > 180338$.
\end{proof}

\begin{theorem}
\label{thm:no-one-perfect-codes-optimized}
There are no $1$-perfect codes in $\textup{J}(n,w)$.
\end{theorem}
\begin{proof} Suppose a $1$-perfect code exists.
By \cref{lem:s7-large}, $\log \nminusw<125000$. The localization  $2^{s_2}\le M$ gives
$$
s_2 \log2\le\log M=\log(\nminusw+d)<\log \nminusw +1<125001<180339\log2.
$$
As a consequence we have that $s_2 \leq 180338$, contradicting \cref{prop:s7-bounded}.
\end{proof}

\section{There Are No \texorpdfstring{$2$}--Perfect Codes}
\label{sec:no_2_perfect}

This section rules out $2$-perfect codes in the Johnson scheme.
A proof with the previous techniques can be provided, but we will give a shorter proof using continued fractions.

A \emph{simple continued fraction} is an expression of the form
$$
a_0 + \cfrac{1}{a_1 + \cfrac{1}{a_2 + \cfrac{1}{a_3 + \dots}}}\triangleq[a_0,a_1,a_2,a_3,\ldots]
$$
where $a_i$ is a positive integer for $i \geq 1$, and $a_0$ can be any integer.
%It is commonly abbreviated using bracket notation:\[[a_0,a_1,a_2,a_3,\ldots].\]
It is known that any positive rational number can be expressed as a finite simple continued fraction (Theorem~161 in \cite{HaWr79}),
and any positive irrational number can be expressed as an infinite simple continued fraction (Theorem~170 in \cite{HaWr79}).
For any real number $x$, the continued fraction algorithm (Section 10.6 in \cite{HaWr79})
defines a finite or infinite sequence of integers $a_0,a_1,a_2,\ldots$ as follows:
\begin{equation}\label{eq:sequence0}
a_0=\lfloor x\rfloor, \xi_0=x-a_0, 0\leq \xi_0<1.
\end{equation}
If $\xi_0\neq 0$, write
\begin{equation}\label{eq:sequence1}
 a_1'=\frac{1}{\xi_0}, a_1=\lfloor a_1'\rfloor, \xi_1=a_1'-a_1, 0\leq \xi_1<1.\\
\end{equation}
If $\xi_1\neq 0$, write
\begin{equation}\label{eq:sequence2}
 a_2'=\frac{1}{\xi_1}, a_2=\lfloor a_2'\rfloor, \xi_2=a_2'-a_2, 0\leq \xi_2<1\\
\end{equation}and so on. The algorithm continues so long as $\xi_n\neq 0$. If the algorithm terminates at some $n$, for which $\xi_n=0$, then
\[x=[a_0,a_1,a_2,a_3,\ldots,a_n],\] which is rational. Otherwise, $x$ is irrational.

The truncated continued fraction at some point $n$, $[a_0,a_1,a_2,a_3,\ldots,a_n]$, is called  the $n$th convergent to $[a_0,a_1,a_2,a_3,\ldots]$, which can be computed by the following proposition.

\begin{proposition}[Theorem 149 and Theorem 155 in \cite{HaWr79}]
\label{thm:convergent}
If $p_n$ and $q_n$ are defined by
\begin{equation*}
\begin{aligned}
&  p_0=a_0, ~p_1=a_1a_0+1,~ p_n=a_np_{n-1}+p_{n-2} ~(n\ge 2);\\
&  q_0=1,  ~~q_1=a_1,~~~~~~~~~ q_n=a_nq_{n-1}+q_{n-2} ~(n\ge 2),
\end{aligned}
\end{equation*}
then \[[a_0,a_1,a_2,a_3,\ldots,a_n]=\frac{p_n}{q_n}.\]
Further $q_n\geq q_{n-1}>0$ for $n\geq 1$.
\end{proposition}

\begin{proposition}[Legendre's criterion, Theorem~184 in \cite{HaWr79}]
\label{thm:quoted-legendre}
Let $z$ be real and let $\frac{x}{y}$ be a reduced rational number with $y\geq 1$.
If
$$
\left|\frac{x}{y}-z\right|<\frac1{2y^2},
$$
then \(x/y\) is a convergent of the simple continued fraction of $z$.
\end{proposition}

\begin{proposition}[Proposition~3 in \cite{EtSc04}]
\label{prop:quoted-etzion-radius2}
If ${n\le40000}$, then there are no $2$-perfect codes in $\textup{J}(n,w)$.
\end{proposition}

\begin{theorem}
\label{thm:section5-two}
There is no nontrivial $2$-perfect code in any Johnson scheme.
\end{theorem}
\begin{proof}
By \cref{lem:6divides} we have that $2$ and $3$ divide $\Phi_2(n,w)$. Choose localized powers
$$
{\bar\omega}<2^u,3^v\le {\bar\omega}+d.
$$
By \cref{prop:quoted-etzion-radius2}, no $2$-perfect code exists for $n\le40000$; hence we assume ${\bar\omega}\ge n/2>20000$, so $v\ge10$.
For the upper bound, \cref{lem:section5-general-cap} gives $v<8115$.

We next apply the continued fraction argument to give possible values of $u/v$ as follows.
Since ${3^v\le M<2{\bar\omega}}$, we have ${\bar\omega}>3^v/2$ and hence by $|\Lambda|<2/\sqrt {\bar\omega}$ from \cref{eq:section5-Lambda-upper},
\begin{equation*}
%\label{eq:section5-Legendre-two}
\left|\frac{\log3}{\log2}-\frac uv\right|= \frac{|\Lambda|}{v\log2}<\frac{2}{v\log2\cdot\sqrt {\bar\omega}} <\frac{2\sqrt2}{v\log2\cdot3^{v/2}}.
\end{equation*}
Since $v\ge 10$ and the function $3^{v/2}/v$ is increasing for $v\ge2$, we have $4\sqrt2v<\log 2\cdot 3^{v/2}$. Hence
\begin{equation}
\label{eq:section5-Legendre-two}
\left|\frac{\log3}{\log2}-\frac{u}{v}\right|<\frac{1}{2v^2}.
\end{equation}
Note that from the proof of \cref{clm:smallgcd}, when $e=2$, we have $\gcd(u,v)= 1$ since $|x-y|\ge1$ and $\theta =|x^g-y^g| >g\bar\omega^{1-1/g}\ge 2\sqrt{\bar\omega}\ge d$. Hence, $\frac{u}{v}$ is a reduced fraction.
Applying Legendre's criterion  in \cref{thm:quoted-legendre} with $x=u,\ y=v$ and $z=\frac{\log3}{\log2}$,
we get that $\frac{u}{v}$ is a convergent of $\frac{\log3}{\log2}$. Since $10\le v<8115$,
$\frac{u}{v}$ can only be some convergent with denominator lying in the interval $[10,8115)$.

By the continued fraction algorithm in \eqref{eq:sequence0}, \eqref{eq:sequence1}, \eqref{eq:sequence2},
the required  prefix of the simple continued fraction of $\frac{\log3}{\log2}$ is
$$
\frac{\log3}{\log2}=[1;1,1,2,2,3,1,5,2,23,\ldots].
$$
By \cref{thm:convergent},  the positive convergents whose denominators
lie in $[10,8115)$ are $\frac{p_i}{q_i}$ with $i=4,5,6,7,8$, which are collected in
$$
D=\{%1/1,\ 2/1,\ 3/2,\ 8/5,\
19/12,\ 65/41,\ 84/53,\ 485/306,\ 1054/665\}.
$$
These are all the required convergents since $q_3=5$, $q_9=15601>8115$ and $q_n\geq q_{n-1}$ (\cref{thm:convergent}).
That is, $u/v$ must belong to $D$.

Finally, we show that $u/v$ cannot equal any element of $D$. Observe that
$$
R_{u,v} \triangleq \frac{|2^u-3^v|^2}{4\min\{2^u,3^v\}}< \frac{d^2}{4\bar{\omega}}\leq \frac{4w}{4\bar{\omega}}\leq 1.
$$
If $u/v$ is some displayed convergent in $D$,
direct substitution gives $R_{u,v}>1$. This contradiction excludes $e=2$.
\end{proof}

\noindent
{\bf Proof of \cref{thm:main_theorem}}
By \cref{thm:section5-unified-tail} we have that there is no $e$-perfect code in J$(n,w)$ when $e \geq 3$.
By \cref{thm:section5-two} there is no $2$-perfect code in J$(n,w)$.
By \cref{thm:no-one-perfect-codes-optimized} there is no $1$-perfect code in J$(n,w)$.
This completes the proof of the main theorem.
\hfill\quad $\Box $

\begin{center}
{\bf \Large{Appendices}}
\end{center}

\appendices

\section{A Proof That ${\bf {w > e^4}}$ for Eight Radii in Lemma~\ref{lem:section5-weight-amplification}}
\label{sec:weight-amplification-supplement}

\begin{lemma}
\label{lem:weight-amplification-supplement}
If $e\in \{4,5,6,9,12,13,20,21\}$ and an $e$-perfect code exists in \textup{J}$(n,w)$, then $w > e^4$.
Moreover, $\bar\omega>15125$ for $e=4$, $\bar\omega>16171$ for $e=5$ and $\bar\omega>2124$ for $e=6$.
\end{lemma}
\begin{proof}
Suppose for the contradiction that $w \le e^4$. Choose from
\cref{lem:compPrimesSmall} the largest prime divisor $q_e$ for each~$e$,
while \cref{eq:qe2-upper} and \cref{eq:qe4-lower} hold for each $q_e$. Therefore $s_{q_e}=3$, so by \cref{cor:prime factors from chain divisitilities}
$q_e^3$ is contained in $(\nminusw,M]$. Since by \cref{cor:str_redund} $d\le e\sqrt w$, it follows that
\begin{equation}\label{eq:qe3-twobounds}
\nminusw=M-d\ge M-e\sqrt{\nminusw}>q_e^3-eq_e^{3/2}, \qquad M=\nminusw+d<q_e^3+e\sqrt{\nminusw}< q_e^3+eq_e^{3/2}.
\end{equation}
We distinguish between seven cases (for the eight radii) depending on the value of $e$.
\begin{itemize}
\item For $e=9$, \cref{lem:compPrimesSmall} yields the prime divisors $5$ and $11$. Let $q_9=11$ and hence
\cref{eq:qe3-twobounds} gives $\bar\omega>1002>5^4$ and $M<1660<5^5$.
Therefore, no power of $5$ lies in the interval $(\nminusw,M]$, contradicting the localization and thus $w > 9^4$.

\item For $e=13$, \cref{lem:compPrimesSmall} yields the prime divisors $7$ and $17$. Let $q_{13}=17$ and hence
\cref{eq:qe3-twobounds} gives $\bar\omega>4001>7^4$ and $M<5825<7^5$.
Therefore, no power of $7$ lies in the interval $(\nminusw,M]$, contradicting the localization and thus $w > 13^4$.

\item For $e=20,21$, \cref{lem:compPrimesSmall} yields the prime divisors $11$ and $23$. Let $q_{20},q_{21}=23$ and hence
\cref{eq:qe3-twobounds} gives $\bar\omega>9850>11^3$ and $M<14484<11^4$.
Therefore, no power of $11$ is contained in the interval $(\nminusw,M]$, contradicting the localization and thus $w > e^4$.

\item For $e=4$, \cref{lem:compPrimesSmall} yields the prime divisor $5$.
Let $q_4=5$ and hence \cref{eq:qe3-twobounds} gives  $\nminusw\ge81$ and $M\le169$.
It implies that the only possible localized powers with exponent larger than $1$ are $2^7=128,5^3=125,11^2=121,13^2=169$.
We claim that $13\nmid\Phi_4(n,w)$; otherwise, both $13^2$ and $5^3$ are contained in $(\nminusw,M]$, and therefore
$$
44=13^2-5^3<d\le 4\sqrt{\bar\omega}<4\sqrt{125}<45.
$$
This is impossible because $d$ is an integer, so $13\nmid\Phi_4(n,w)$.
Consider the product of all prime divisors of $\Phi_4(n,w)$ with exponent larger than $1$ and denote this product by $\Upsilon$. Then
\begin{equation}\label{eq:Phi'-upper}
\Upsilon \le 2\cdot5\cdot11.
\end{equation}
From the proof of \cref{lem:number_mult1}, prime divisors with localized exponent $1$ can only be $\bar\omega+h_i+1$ for some $h_i$. Hence, \cref{thm:section5-Lloyd-factorization}, together with \eqref{eq:section5-w-over-z} and $d\le e\sqrt w$, gives
\begin{align}\label{eq:section5-e4-product}
\Upsilon \ge \frac{\prod_{i=1}^4(w-h_i)}{(4!)^2}\ge \frac{(w-d)^4}{576}> \frac{(4\bar\omega/5-4\sqrt{\bar\omega})^4}{576}=
\frac{4}{9} \Bigl( \frac{\bar\omega}{5}-\sqrt{\bar\omega} \Bigr)^4.
\end{align}
Since $\bar\omega\ge 81$, we have $\Upsilon>1194>2\cdot5\cdot11$, contradicting \cref{eq:Phi'-upper} and
thus $w > 4^4$. This implies that $s_5\ge4$.
We continue to exclude larger $s_5$. By applying \cref{eq:qe3-twobounds} with any given $s_5$, we have that
$$
\nminusw> 5^{s_5}-4\cdot5^{s_5/2},\qquad M< 5^{s_5}+4\cdot5^{s_5/2}.
$$
If $s_5=4$, then $\nminusw>525$ and $M<725$, producing the only possible localized powers $5^4, 23^2$ with exponent larger than $1$; however, \cref{eq:section5-e4-product} gives $\Upsilon>2\cdot10^7>5\cdot23$, a contradiction.
If $s_5=5$, then $\nminusw>2901$ and $M<3349$, producing the unique localized power $5^5$ with exponent larger than $1$;
\cref{eq:section5-e4-product} again gives a contradiction. Therefore, $s_5\ge6$ and thus $\nminusw>5^6-4\cdot5^3=15125$.

\item For $e=5$, \cref{lem:compPrimesSmall} yields the prime divisors $2,7$.
Let $q_5=7$ and hence \cref{eq:qe3-twobounds} gives $\nminusw>250$ and $M<436$.
It implies that the localized power of $2$ must be $2^8=256$. Therefore, $7^3$ and $2^8$ are contained in the localized
interval $(\nminusw,M]$, and it implies that
$$
87=7^3-2^8<d\le5\sqrt {\bar\omega}<5\sqrt{256}=80,
$$
a contradiction, and thus $w > 5^4=625$. Moreover, this implies that the localized powers are at least $2^{10}=1024$ and $7^4=2401$.
Therefore, $\bar\omega=M-d\ge2401-5\sqrt{\bar\omega}$. Solving this inequality gives $\bar\omega>2168$.
Iteratively the localized power of $2$ is at least $2^{12}=4096$, and thus $\bar\omega=M-d\ge4096-5\sqrt{\bar\omega}$;
solving this inequality gives $\bar\omega>3788$.
Iteratively the localized power of $7$ is at least $7^5=16807$, and thus $\bar\omega=M-d\ge16807-5\sqrt{\bar\omega}$;
solving this inequality gives $\bar\omega>16171$.

\item For $e=6$, \cref{lem:compPrimesSmall} yields the prime divisors $2,7$.
Let $q_6=7$ and hence \cref{eq:qe3-twobounds} gives  $\nminusw>231$ and $M<455$.
It implies that the localized power of $2$ must be $2^8=256$ and therefore $231<\bar\omega<256$.
Moreover, the residue localization from \cref{lem:residue-localization} with $(p,P,\gamma)=(7,7,1)$ and $(p,P,\gamma)=(2,8,2)$, respectively, gives $\bar\omega\equiv6\pmod7$ and $\bar\omega\equiv5\text{ or }6\pmod8$, respectively, i.e., $\nminusw = 237$. However,
$$
343=7^3 \leq M = \nminusw +d < \nminusw + 6\sqrt{\nminusw}= 237 +6 \sqrt{237} < 330.
$$
This contradiction establishes that $w>6^4=1296$. Moreover, this implies that the localized powers
are at least $2^{11}=2048$ and $7^4=2401$. Therefore,
$\bar\omega=M-d\ge2401-6\sqrt{\bar\omega}$. Solving this inequality gives $\bar\omega>2124$.

\item For $e=12$, \cref{lem:compPrimesSmall} yields the prime divisors $2,7$, and $13$.
Let $q_{12}=13$ and hence \cref{eq:qe3-twobounds} gives  $\nminusw>1634$ and $M<2760$.
It implies that the localized powers of $2$ and~$7$ must be $2^{11}=2048,\ 7^4=2401$, respectively. As a consequence
$$
\nminusw<2048,\qquad {\bar\omega}=M-d>7^4-12\sqrt{2048}>1857.
$$
We now use the residue localization already contained in \cref{lem:residue-localization}. For the three
parameters $(p,P,\gamma)=(13,13,1),(7,7,2),(2,16,4)$ given by \cref{lem:modulo-p-criterion},
$$
\nminusw\equiv12\pmod{13},\qquad \nminusw \equiv4\text{ or }5\pmod7,\qquad {\bar\omega}\equiv9,10,11,\text{ or }12\pmod{16}.
$$
\cref{eq-congruence-e+2} with $e+2=14$ sharpens the congruence modulo $7$ to $\nminusw \equiv5\pmod7$.
A simple calculation gives the possible residues modulo $13\cdot7\cdot16=1456$. Writing $\bar\omega=12+13r$,
the congruence modulo $7$ gives $r\equiv0\pmod7$, while the congruence modulo $16$ gives $r\equiv0,1,6,11\pmod{16}$.
Combining these modulo $7\cdot16=112$ yields
$$
r\equiv0,49,70,91\pmod{112},
$$
and hence
$$
{\bar\omega}\pmod{1456}\in\{12,649,922,1195\}.
$$
None of these residue classes has a representative in $1857<{\bar\omega}<2048$, a contradiction and thus $w > 12^4$.
\end{itemize}
\end{proof}

\section{Supplemental Material}
\label{sec:added}

\begin{center}
{\Large\bfseries Computer-assisted verification for \cref{lem:s7-certificate}}
\end{center}

Algorithm~\ref{alg:lemma810-certificate}, together with its accompanying program, verifies \cref{lem:s7-certificate} using only integer arithmetic.
The program completed the verification in less than $2$ minutes.
Here $x\bmod N$ denotes the least nonnegative remainder modulo $N$.

\begin{algorithm}[H]
\small
\DontPrintSemicolon
\SetKwInOut{Input}{Input}
\SetKwInOut{Output}{Output}
\SetKw{Return}{return}
\Input{$L=250$, $H=180338$, $B=40$.}
\Output{\texttt{PASS} if all tests succeed; otherwise \texttt{INCONCLUSIVE}.}
\ForEach{odd prime $p\le 173$ (sieve of Eratosthenes)}{
  $V\gets\lfloor H/p\rfloor+B+1$\;
  \For{$r\gets1$ \KwTo $p-1$}{
    $u_0\gets L+((r-L)\bmod p)$\;
    $A\gets\lfloor(2^{pV+r})^{1/p}\rfloor$ (exact integer root)\;
    \For{$u=u_0,u_0+p,\ldots\le H$}{
      $s\gets V-\lfloor u/p\rfloor$\;
      $a\gets\lfloor A/2^s\rfloor$\;
      $f\gets\lfloor A/2^{s-B}\rfloor\bmod2^B$\;
      \If{$(a\text{ odd}\ \land\ f=0)\ \lor\ (a\text{ even}\ \land\ f=2^B-1)$}{
        \Return{\texttt{INCONCLUSIVE}}\;
      }
    }
  }
}
\Return{\texttt{PASS}}\;
\caption{Exact finite verification for Lemma~\ref{lem:s7-certificate}.}
\label{alg:lemma810-certificate}
\end{algorithm}

To justify the test, suppose for a contradiction that \eqref{eq:s7-finite-gap} holds, and put $t=2^{u/p}$. Then $|2^u-m^p|<2^{u/2}$ and $m^p>2^u-2^{u/2}>2^{u-1}$.
Thus $\min\{t,m\}>2^{(u-1)/p}$, and the mean-value theorem yields
\[
\begin{aligned}
|t-m|
&\le \frac{|t^p-m^p|}{p\min\{t,m\}^{p-1}}<\frac{2^{u/2}}{p\min\{t,m\}^{p-1}}<\frac{2^{u/2-(u-1)(p-1)/p}}p\\
&\le\frac{2^{(4-u)/6}}p<2^{-40},
\end{aligned}
\]
where $p\ge3$ and $u\ge250$, and the fourth inequality holds because the exponent $u/2-(u-1)(p-1)/p$ is decreasing for $p\ge3$.
If $p\mid u$, then $t=2^{u/p}$ is an even integer, contradicting
$|t-m|<1$ for odd $m$. Otherwise write
$u=p\lfloor u/p\rfloor+r$ with $1\le r<p$. The exact integer root
in the algorithm gives
$$
A=\lfloor2^{V+r/p}\rfloor
 =\lfloor2^s t\rfloor,
\qquad s=V-\lfloor u/p\rfloor\ge B+1.
$$
Using $\lfloor\lfloor x\rfloor/2^k\rfloor=\lfloor x/2^k\rfloor$
for integers $k\ge0$, we obtain
\[
a:=\lfloor A/2^s\rfloor =\lfloor t\rfloor,\qquad
f:=\lfloor A/2^{s-B}\rfloor\bmod2^B =\lfloor2^B\{t\}\rfloor,
\quad\text{where }\{t\}=t-\lfloor t\rfloor.
\]
If $a$ is odd, the only odd integer $m$ within distance $1$ of $t$
is $a$. The test $f\ne0$ implies $t-a\ge2^{-B}$.
If $a$ is even, the only possible $m$ is $a+1$, and
$f\ne2^B-1$ implies $a+1-t>2^{-B}$.
Thus a \texttt{PASS} excludes $|t-m|<2^{-40}$ in every case,
proving \cref{lem:s7-certificate}.

\begin{center}
{\Large\bfseries The Program}
\end{center}

{\footnotesize
\begin{verbatim}
"""Exact certificate for Lemma 7.7; requires SymPy."""
from math import isqrt
from sympy import integer_nthroot

L, H, B = 250, 180338, 40
MASK = (1 << B) - 1

def primes(n):
    s = bytearray(b'\1')*(n+1); s[:2] = b'\0\0'
    for p in range(2, isqrt(n)+1):
        if s[p]: s[p*p:n+1:p] = b'\0'*len(range(p*p, n+1, p))
    return [p for p in range(2, n+1) if s[p]]

P = primes(173)

for q in (p for p in P if 3 <= p <= 173):
    V = H//q + B + 1
    for r in range(1, q):
        first = L + (r-L) % q
        if first > H: continue
        A, exact = integer_nthroot(1 << (q*V+r), q)
        assert not exact
        for u in range(first, H+1, q):
            s = V - u//q
            a = A >> s
            f = (A >> (s-B)) & MASK
            brackets = (a, a+2) if a & 1 else (a-1, a+1)
            assert f != 0 if brackets[0] == a else f != MASK

print("PASS")
\end{verbatim}
}

\end{document}